\documentclass[11pt, reqno]{amsart}
\usepackage{color}

\usepackage{amsmath}
\usepackage{amssymb}
\usepackage{amsfonts}
\usepackage{mathrsfs}
\usepackage{amsbsy}
\usepackage{relsize}
\usepackage[colorlinks, linkcolor=red, citecolor=blue, urlcolor=blue, pagebackref, hypertexnames=false]{hyperref}
\usepackage[hyperpageref]{backref}

\usepackage{tikz}

\numberwithin{equation}{section}

\numberwithin{bbb}{section}

\newtheorem{th1}{{\bf Theorem}}[section]
\newtheorem{thm}[th1]{{\bf Theorem}}
\newtheorem{lem}[th1]{{\bf Lemma}}
\newtheorem{prop}[th1]{{\bf Proposition}}
\newtheorem{cor}[th1]{{\bf Corollary}}
\newtheorem{rem}[th1]{\bf Remark}

\newtheorem{defs}[th1]{\bf Definition}

\newcommand{\R}{\mathbb{R}}
\newcommand{\Z}{\mathbb{Z}}

\newcommand{\C}{\mathbb{C}}
\newcommand{\bS}{\mathbf{S}}
\newcommand{\bB}{\mathbf{B}}
\newcommand{\bI}{\mathbf{I}}
\newcommand{\bJ}{\mathbf{J}}
\newcommand{\bA}{\mathbf{A}}

\title[INLS]{Well-posedness and scattering for a 2D inhomogeneous NLS with Aharonov-Bohm magnetic potential}

\author[M. Majdoub and T. Saanouni]{Mohamed Majdoub and Tarek Saanouni}
\address[M. Majdoub]{Department of Mathematics, College of Science, Imam Abdulrahman Bin Faisal University, P. O. Box 1982, Dammam, Saudi Arabia.\newline Basic and Applied Scientific Research Center, Imam Abdulrahman Bin Faisal University, P.O. Box 1982, 31441, Dammam, Saudi Arabia.}
\email{\sl mmajdoub@iau.edu.sa}
\email{\sl med.majdoub@gmail.com}
\address[T. Saanouni]{ Departement of Mathematics, College of Science and Arts in Uglat Asugour, Qassim University, Buraydah, Kingdom of Saudi Arabia.}
\email{\sl t.saanouni@qu.edu.sa}


\begin{document}
\begin{abstract}
We consider the magnetic nonlinear inhomogeneous Schr\"odinger equation
$$i\partial_t u -\left(-i\nabla+\frac{\alpha}{|x|^2}(-x_2,x_1)\right)^2 u =\pm|x|^{-\varrho}|u|^{p-1}u,\quad (t,x)\in \mathbb{R}\times \mathbb{R}^2,$$
where $\alpha\in\mathbb{R}\setminus\mathbb{Z},\,\varrho>0,\,p>1$. We prove a dichotomy of global existence and scattering versus blow-up of energy solutions under the ground state threshold in the inter-critical regime. The scattering is obtained by using the new approach of Dodson-Murphy (A new proof of scattering below the ground state for the 3D radial focusing cubic NLS, {Proc. Am. Math. Soc.} (2017)). This method is based on Tao's scattering criteria and Morawetz estimates. The novelty here is twice: we investigate the case $\varrho\alpha\neq0$ and we consider general energy initial data (not necessarily radially symmetric). The particular case $\alpha=0$, known as INLS, was widely investigated in the few recent years. Moreover, the particular case $\varrho=0$, which gives the homogeneous regime, was considered recently by X. Gao and C. Xu (Scattering theory for NLS with inverse-square potential in 2D, J. Math. Anal. Appl. (2020)), where the scattering is proved for spherically symmetric datum. In the radial framework, the above problem translate to the INLS with inverse square potential, which was widely investigated in space dimensions higher than three. The Hardy inequality $\||x|^{-1}f\|_{L^2(\mathbb{R}^N)}\leq \frac{2}{N-2}\|\nabla f\|_{L^2(\mathbb{R}^N)}$, which gives the norm equivalence 
$\|f\|_{H^1}\simeq \|f\|_{H^1}+\||x|^{-1}f\|_{L^2}$, fails in two space dimensions. Thus, it is not clear how to treat the NLS with inverse square potential in $H^1$ for two space dimensions. This article seems to be the first one dealing with the NLS with Aharonov-Bohm magnetic potential in the inhomogeneous regime, namely $\varrho\neq0$. 
\end{abstract}
\date{\today}
\maketitle

\section{Introduction and main results}
\label{S1}
We consider the initial value problem for the nonlinear inhomogeneous Schr\"odinger equation with Aharonov-Bohm potential
\begin{equation}\label{S}
\begin{cases}
i\partial_t u -\left(-i\nabla+\frac{\alpha}{|x|^2}(-x_2,x_1)\right)^2 u=\kappa|x|^{-\varrho}|u|^{p-1}u,   \\
u(t=0,x)= u_0(x),
\end{cases}
\end{equation}
where the wave function is $u:=u(t,x)\in\C$, $t\in\R$ denotes the time variable, $x:=(x_1,x_2)\in\R^2$ is the space variable, $p>1$ is the exponent of the source term, $\varrho>0$ gives a singular inhomogeneous term in the non-linearity, $\alpha\in\R\setminus\Z$ gives an Aharonov-Bohm potential and $\kappa=\pm1$. Moreover, $\kappa=1$ stands for the defocusing case while $\kappa=-1$ corresponds to the focusing regime.

The linear counterpart of \eqref{S} is the so-called electromagnetic Schr\"odinger equation
$$i\partial_t u =\Big(-i\nabla+\frac{{\bf A}(\frac x{|x|})}{|x|}\Big)^2u+\frac{{\bf a}(\frac x{|x|})}{|x|^2} u, $$
where ${\bf a}\in W^{1,\infty}(\bS^1,\R)$, $\bS^1$ denotes the unit
circle, and ${\bf A} \in W^{1,\infty}(\bS^1,\R^2)$ is a transversal vector field, namely
$${\bf A}(\omega)\cdot \omega=0,\quad\forall\,\, \omega\in\bS^1.$$

The magnetic potential (1-
form) $\frac{{\bf A}(\frac x{|x|})}{|x|}$ is connected with the associated magnetic tensor (2-form) ${\bf B}$ by the
exterior derivative ${\bf B} =d\frac{{\bf A}(\frac x{|x|})}{|x|}$. The magnetic tensor is compatible since the Maxwell equation
$d{\bf B} = 0$ means that ${\bf B}$ is a closed form. The Aharonov–Bohm potential gives a magnetic field
associated to thin solenoids: if the radius of the solenoid tends to zero while the flux through
it remains constant, then the particle is subject to a $\delta$-type magnetic field, that is so-called
Aharonov–Bohm field. The Aharonov–Bohm (AB) effect \cite{ab} lies at the interface of
gauge theories and quantum mechanics. In its best known form,
the AB effect predicts a shift in the interference pattern of the
quantum mechanical double-slit experiment which has a magnetic
flux carrying solenoid placed between the slits. If a solenoid with
a magnetic field ${\bf B} =\nabla\times{\bf A}$ (where ${\bf A}$ is the electromagnetic vector
potential) is placed between the two slits of a double-slit experiment the phase of the wave-function of the electrons going
through the slits and following some path to the screen \cite{dsecv,jzwld}. We refer to \cite{ahs} for a rigorous study of the magnetic field in quantum mechanics and to \cite{erd, PhysRev, IET} for other physical aspects with many references therein.\\

The vector potential
\begin{equation}
\label{A-alpha}
\bA_{\alpha}(x):=\frac{\alpha}{|x|^2}x^\perp,\quad x^\perp:=(-x_2,x_1)\in\R^2\setminus\{0\},
\end{equation}
generates the Aharonov–Bohm magnetic field. We know from \cite{AT98} that the operator 
\begin{equation}
\label{K-alpha}
\left(-i\nabla+\bA_{\alpha}\right)^2 \quad\mbox{on}\quad C_0^\infty(\R^2\setminus\{0\}) 
\end{equation}
is not essentially self-adjoint and admits  infinitely many self-adjoint extensions. Here we choose to work with the Friedrichs extension of \eqref{K-alpha} denoted by $\mathcal{K}_\alpha$. To be more precise, we introduce the  space $\dot H_\alpha^1$ as the closure of $C_0^\infty(\R^2\setminus\{0\})$ with respect to the norm

$$\big\|\left(\nabla+i\bA_{\alpha}\right)u\big\|_{L^2(\R^2)}.$$
We also define the associated inhomogeneous space $H^1_\alpha(\R^2):=L^2(\R^2)\cap \dot H_\alpha^1(\R^2)$. The quadratic form on $H^1_\alpha(\R^2)$ given by
$$
u\longmapsto \big\|\left(\nabla+i\bA_{\alpha}\right)u\big\|_{L^2(\R^2)}^2
$$
is closed and generates a unique non-negative self-adjoint operator $\mathcal{K}_\alpha$ on $L^2(\R^2)$ with domain 
$$
\mathcal{D}=\left\{f\in H^1_\alpha(\R^2);\;\;\; \mathcal{K}_\alpha f\in L^2(\R^2)\right\}.
$$
Moreover, the unitary group ${\rm e}^{it \mathcal{K}_\alpha}$ extends to a group of isometries on the dual $\mathcal{D}^*$ of $\mathcal{D}$. Therefore for every $u_0\in L^2(\R^2)$, the unique solution to \eqref{E-Lin} reads
$$
u(t,x):={\rm e}^{it \mathcal{K}_\alpha}\,u_0(x)\in C(\R; L^2(\R^2))\cap C^1(\R; \mathcal{D}^*).
$$

 
Note that the operator $\mathcal K_\alpha$ acts on functions as follows
\begin{equation*}
\label{K-radial}
\mathcal K_\alpha u:=-\Delta u+\frac{\alpha^2}{|x|^2}u-2i\frac{\alpha}{|x|^2}x^\perp\cdot\nabla u=-\Delta u+\frac{\alpha^2}{|x|^2}u,
\end{equation*}
where the later equality is valid in the spherically symmetric framework.  

From \cite[Remark 2.1, p. 3891]{GK-JDE-2014} and \cite[(4.1), p. 3895]{GK-JDE-2014}, we know that $\mathcal{K}_\alpha$ and $\mathcal{K}_{\alpha+m}$ are unitary equivalent for any $m\in\Z$. Hence, we may assume without loss of generality that $0<|\alpha|\leq\frac{1}{2}$. 

 The following Hardy inequality was obtained in \cite[Theorem 3]{lw}:
\begin{equation}
\label{Ha-In}
\left(d(\alpha, \Z)\right)^2\,\int_{\R^2}\,\frac{f(x)}{|x|^2}\,dx\leq \int_{\R^2}\,|\nabla_{\alpha}f(x)|^2\,dx,\quad \alpha\in \R\setminus\Z,
\end{equation}
where
\begin{equation}
\label{nab-alpha}
\nabla_\alpha:=\nabla+i\bA_{\alpha}=\nabla+i\frac{\alpha}{|x|^2}x^\perp.
\end{equation}
 From \eqref{nab-alpha} and \eqref{Ha-In}, we obtain the following Sobolev embedding
 \begin{equation}
\label{Sob-emb}
H^1_\alpha(\R^2)\hookrightarrow H^1(\R^2) \hookrightarrow L^r(\R^2),\;\; 2\leq r<\infty,\;\; \alpha \in \R\setminus\Z.
 \end{equation}

The dispersive and Strichartz estimates are fundamental tools in studying the linear and nonlinear dynamics for dispersive equations. In our setting, we refer to \cite{fff2,gyzz}. See also \cite[Theorem 2.3, p. 91]{Fanelli} for a precise statement of the dispersive estimate for the Aharonov-Bohm potential. It is now quite classical that the dispersive estimate implies Strichartz one by  means of the $TT^*$ argument of Keel-Tao \cite{keel}. See, among many, \cite{gyzz, GK-JDE-2014} and the references therein. For the case of Schr\"odinger operator with inverse-square potential, we refer to  \cite{TAMS22}. The Strichartz estimates for the Aharonov-Bohm potential was successfully used to obtain well-posedness and scattering for the homogeneous case, that is \eqref{S} with $\varrho=0$. In  \cite{zz}, the authors prove the scattering in the defocusing regime for radially symmetric initial data under the supplementary conditions $p>3$ and $2\alpha^2>1$. The Scattering in the focusing regime was obtained in \cite{gx} for radially symmetric initial data under the ground state threshold.

 Now, we turn back to \eqref{S}.  The main  purpose of this work is to extend and somehow improve the results in \cite{gx,cg} to the case of singular weight in the nonlinearity, that is $\varrho>0$.

 Since the inverse-square potential $|x|^{-2}$ conserves the same scaling of the Laplace operator, the non-linear Schr\"odinger equation \eqref{S} satisfies the scaling invariance
$$0<\mu\longmapsto u_\mu(t,x):=\mu^{\frac{2-\varrho}{p-1}}u(\mu^{2}t,\mu x).$$
The identity $\|u_\mu(t)\|_{\dot H^s}=\mu^{s-(1-\frac{2-\varrho}{p-1})}\|u(\mu^{2}t)\|_{\dot H^s}$ gives the only one homogeneous Sobolev norm stable under the above dilatation. It is $s_c:=1-\frac{2-\varrho}{p-1}$ called the critical Sobolev index. The energy-critical case corresponds to $s_c=1$, or $p=\infty$. This case is related to the energy conservation law
\begin{align*}
E[u(t))]&:=\int_{\R^2}|\nabla_\alpha u(t)|^2\,dx+\kappa \frac2{p+1}\int_{\R^2}|x|^{-\varrho}|u(t)|^{p+1}\,dx = E[u_0].\tag{Energy}
\end{align*}
The mass-critical one corresponds to $s_c=0$, or $p=p_c:=3-\varrho$ which is related to mass conservation law 
\begin{align*}
M[u(t)]&:=\int_{\R^2}|u(t,x)|^2\,dx = M[u_0].\tag{Mass}
\end{align*}
 It is worth to mention that $s_c<1$ provided that $\varrho<2$. This means that \eqref{S} is energy sub-critical for any $p>1$. 

In the sequel we will focus on the inter-critical regime $0<s_c<1$, that is $3-\varrho<p<\infty$, and define the positive real number
\begin{equation}
\label{gam-c}
\lambda_c:=\frac{1}{s_c}-1=\frac{p_c-1}{p-p_c}=\frac{2-\varrho}{p+\varrho-3}.
\end{equation}
Here and hereafter, one denotes for simplicity the Lebesgue norms
$$\|\cdot\|_r:=\|\cdot\|_{L^r(\R^2)}\quad\mbox{and}\quad \|\cdot\|:=\|\cdot\|_2.$$
Define also the quantities
\begin{eqnarray}
\label{P-u}
\mathcal P[u]&:=&\int_{\R^2}|x|^{-\varrho}|u|^{p+1}\,dx,\\
\label{Q-u}
\mathcal Q[u]&:=&\|{\nabla_\alpha}\, u\|^2-\frac{B}{p+1}\mathcal P[u]:=\|{\nabla_\alpha}\, u\|^2-\frac{p-1+\varrho}{p+1}\mathcal P[u].
\end{eqnarray}
 Denote the scale invariant quantities
\begin{align}
\label{M-E}
\mathcal{EM}[u]&:=\Big(\frac{E[u]}{E[\phi]}\Big)\Big(\frac{M[u]}{M[\phi]}\Big)^{\lambda_c},\\
\label{M-G}
\mathcal{GM}[u]&:=\Big(\frac{\|\nabla u\|}{\|\nabla\,\phi\|}\Big)\Big(\frac{M[u]}{M[\phi]}\Big)^\frac{\lambda_c}2,\\
\label{M-P}
\mathcal{PM}[u]&:=\Big(\frac{\mathcal P[u]}{\mathcal P[\phi]}\Big)\Big(\frac{M[u]}{M[\phi]}\Big)^{\lambda_c},
\end{align}
where $\phi$ is a ground state solution to \eqref{gr}.

From now one hides the variable t for simplicity, spreading it out only when necessary.

Our main contribution reads as follows.
\begin{thm}\label{t1}
Let $0<\varrho<1$, $3-\varrho<p<\infty$, $\alpha\in \R\setminus\Z$ and $u_0\in H_\alpha ^1$. Let $\phi$ be a ground state solution to \eqref{gr} and ${u}\in C([0,T^*), H_\alpha ^1)$ be the maximal solution  of \eqref{S} given by Theorem \ref{loc} below. 
\begin{enumerate}
\item[1)]
Suppose that 
\begin{equation}
\sup_{t\in[0,T^\ast)}\mathcal{PM}[u(t)]<1.\label{ss1}
\end{equation}
Then, $u$ is global and  scatters in $H^1_\alpha $. 
\item[2)]
Suppose that 
\begin{equation}
\sup_{t\in[0,T^\ast)}\mathcal Q[u(t)]<0.\label{bl}
\end{equation}
Then
\begin{equation}
\label{Blow-infinity}
\sup_{t\in[0,T^\ast)}\,\|\nabla u(t)\|=\infty.
\end{equation}

\end{enumerate}
\end{thm}
In view of the  results stated in the above theorem, some comments arise and we enumerate them in what follows.
~{\rm \begin{itemize}
\item[($i$)] The assumption $\alpha\in \R\setminus\Z$ enables us to use the Hardy estimate \eqref{Ha-In}.
\item[($ii$)]By Proposition \ref{gag}, it follows that the energy is well-defined for $0<\varrho<2$. However, one needs the restriction $0<\varrho<1$ in the local theory. This is due to the method based on using a fix point argument and Strichartz estimates. Moreover, in order to control the source term, one decomposes the integrals on the unit ball of $\R^2$ and its complementary. In a paper in progress, the authors try to improve the range of the inhomogeneous term exponent $\varrho$ by use of Lorentz spaces in the spirit of \cite{last}.
\item[($iii$)] The scattering under the ground state threshold, in the spirit of the pioneering work \cite{Merle}, is a consequence of the above result. This is given in Corollary \ref{t2} below.
\item[($iv$)]
The two above criteria about scattering and blow-up are expressed in terms of non-conserved quantities in the spirit of \cite{dd}. This makes more difficult to check their availability. But the assumptions \eqref{ss1} and \eqref{bl} are weaker than the classical ones, namely \eqref{t11}-\eqref{t12} and \eqref{t11}-\eqref{t13}, which are expressed in term of mass and energy and so more simple to check.
\item[($v$)]
The scattering is obtained by using the new approach of Dodson-Murphy \cite{DM2017}. This method is based on Tao's scattering criteria \cite{Tao} and Morawetz estimates.
\item[($vi$)] 
Thanks to the identity $\mathcal{Q}[u]=\frac{B}{2}E[u]-\frac{B-2}{2}\|\nabla_\alpha u\|^2$, the above blow-up result holds for negative energy.
\item[($vii$)] In the homogeneous case $\varrho=0$, the scattering under the ground state threshold with radial data was obtained in \cite{gx}.
\item[($iix$)] Our results extend and improve the ones in \cite{gx, zz}.
\item[($ix$)] It is expected that the energy critical non-linearity  in \eqref{S} should be of exponential type. This fact was confirmed and extensively studied in the last decade for $\alpha=0$, that is the 2D-{\tt NLS} with exponential non-linearity. See \cite{BDM2018, BDM2019, CIMM, DKM, IMMN} and the references therein. For the inverse-square potential and for both {\tt NLS} and Klein-Gordon equations, see \cite{sd}. 
\item[($x$)]  In a paper in progress, the authors treat the problem \eqref{S} with exponential type non-linearity.
\end{itemize}}

As a consequence of the above result, one has the next dichotomy of global/non global existence of energy solutions under the ground state threshold.
\begin{cor}\label{t2}
Take the assumptions of Theorem \ref{t1} and suppose further that
\begin{equation} \label{t11}
\mathcal{EM}[u_0]<1.
\end{equation}
Then, 
\begin{enumerate}
\item[1)]
the solution of \eqref{S} is global and scatters if 
\begin{equation}\label{t12}
\mathcal{GM}[u_0]<1.\end{equation}
\item[2)]
the solution blows-up in finite or infinite time in the sense of \eqref{Blow-infinity} if 
\begin{equation}\label{t13}
\mathcal{GM}[u_0]>1.\end{equation}
\end{enumerate}
\end{cor}

{ The rest of this paper is organized as follows. The next section contains the main results and some standard estimates needed in the sequel. Section 3 develops a local theory in the energy space. In section 4, one proves the main result of this note about two criteria of scattering versus blow-up of energy solutions. The last section proves a dichotomy of global existence and scattering versus blow-up of solutions under the ground state threshold.}  \\

{Finally, for $a\in\R$, one denotes $a^+$ to be a real number close to $a$ such that $a^+>a$ and $a^-$ to be a real number close to $a$ such that $a^-<a$}
\section{Useful tools and auxiliary results}
\label{S3}

For future convenience,  we recall some known and useful tools which will play an important role in the proof of our main results.

First, let us collect some standard estimates related to the linear electromagnetic Schr\"odinger equation:
\begin{equation}\label{E-Lin}
\begin{cases}
i\partial_t u -\mathcal{K}_\alpha u=F(t,x),   \\
u(t=0,x)  =  u_0(x).
\end{cases}
\end{equation}
The following dispersive estimate can be found in \cite{fff1,fff2}.
\begin{lem}\label{dsp}
Let $\alpha\in\R\setminus\Z$, $t\in\R\setminus\{0\}$ and $2\leq r\leq\infty$. Then,
\begin{equation}
\label{Dis-est}
\left\|e^{it\mathcal{K}_\alpha }\right\|_{L^{r'}\to L^r}\lesssim |t|^{\frac{2}{r}-1}.
\end{equation}
\end{lem}
To state the Strichartz estimates, we need the following definition of admissible pairs.
 \begin{defs}
A pair $(q,r)$ is said to be admissible if
\begin{equation}
\label{Adm}
\frac{1}{q}+\frac{1}{r}=\frac{1}{2},\quad q,r\geq 2,\quad (q,r)\neq (2,\infty).
\end{equation}
\end{defs}
Let $\Gamma$ be the set of all admissible pairs. For $T>0$ and $\Omega\subset \R^2$ a measurable set, define 
\begin{eqnarray*}
\|u\|_{S(\Omega, T)}&=&\sup_{(q,r)\in\Gamma}\|u\|_{L^q(0,T; L^r(\Omega))},\\
\|u\|_{S'(\Omega, T)}&=&\inf_{(q,r)\in\Gamma}\|u\|_{L^{q'}(0,T; L^{r'}(\Omega))}.
\end{eqnarray*}
For $\Omega=\R^2$, we simply write
$$
\|u\|_{S(T)}=\|u\|_{S(\R^2, T)},\quad \|u\|_{S'(T)}=\|u\|_{S'(\R^2, T)}.
$$
Thanks to the above disperive estimate \eqref{Dis-est} and an argument of Keel-Tao \cite{keel}, one obtains some Strichartz estimates as stated below.
\begin{prop}[{\cite{gyzz, GK-JDE-2014}}]

\label{prop2}
Let $\alpha\in\R\setminus\Z$, $T>0$ and $u$ the solution of \eqref{E-Lin}. Then
\begin{equation}
\label{Str-est1}
\|u\|_{S(T)}\lesssim \|u_0\|_{L^2}+\|F\|_{S'(T)}.
\end{equation}
\end{prop}
Employing $\left[\nabla_{\alpha}, \mathcal{K}_{\alpha}\right]=0$ together with \eqref{Str-est1}, we infer that
\begin{equation}
\label{Str-est2}
\|\nabla_{\alpha}u\|_{S(T)}\lesssim \|u_0\|_{\dot{H}^1_{\alpha}}+\|\nabla_{\alpha}F\|_{S'(T)},
\end{equation}
\begin{equation}
\label{Str-est3}
\|u\|_{L^\infty_T(H^1_\alpha)}\lesssim \|u_0\|_{H^1_{\alpha}}+\|\langle \nabla_{\alpha}\rangle F\|_{S'(T)},
\end{equation}
where
\begin{equation}
\label{nabla-alpha}
\langle \nabla_{\alpha}\rangle=\sqrt{1+|\nabla_\alpha|^2}.
\end{equation}
We also recall the following local-in-time Strichatz estimate (see, for instance, \cite{cc})
\begin{equation}
\label{Loc-Str}
\left\|\int_a^b\, {\rm e}^{i(t-\tau)\mathcal{K}_\alpha}g (\cdot, \tau) d\tau\right\|_{S(\R)}\lesssim \|g\|_{S'(\R)}.
\end{equation}
The following Gagliardo-Nirenberg inequality will be of interest in the proofs of our main results.
\begin{prop}\label{gag}
Let $\alpha\in \R\setminus\Z$, $1<p<\infty$, $0<\varrho<2$. 
Then the following sharp Gagliardo-Nirenberg inequality holds 
\begin{equation}\label{ineq}
\int_{\R^2}|x|^{-\varrho}|f(x)|^{p+1}\,dx\leq \mathtt{K}_{opt} \|f\|^{A}\|f\|_{\dot H^1_\alpha }^{B},\quad  f\in H_\alpha ^1(\R^2),
\end{equation}
where $A$ and $B$ are given by 
\begin{equation}
\label{AB}
B:=p-1+\varrho\quad\text{and} \quad A:=1+p-B.
\end{equation}
Moreover, the sharp constant $K_{opt}$ is given by
\begin{equation}\label{part3}
\mathtt{K}_{opt} =\frac{p+1}{A}\Big(\frac{A}{B}\Big)^\frac{B}2\|\phi\|^{-(p-1)},
\end{equation}
where $\phi$ is a ground state solution to 
\begin{equation}\label{gr}
\mathcal K_\alpha \phi+\phi=|x|^{-\varrho}|\phi|^{p-1}\phi,\quad 0\neq\phi\in H_\alpha ^1(\R^2).
\end{equation}
\end{prop}
Before getting to the proof of Proposition \ref{gag}, we prove the following compact Sobolev embedding.

\begin{lem}
\label{compact}
Let $p>1$, $0<\varrho<2$ and $\alpha\in\R\setminus\Z$. Then we have  the compact Sobolev embedding
\begin{equation}\label{cpct}
H_\alpha ^1(\R^2)\hookrightarrow\hookrightarrow L^{p+1}(|x|^{-\varrho}\,dx).
\end{equation}
\end{lem}
\begin{proof}
Suppose that $f_n\rightharpoonup0$ in $H_\alpha ^1$ and let $R>\frac{1}{2-\varrho}$. It follows from H\"older's inequality that
\begin{eqnarray*}
\int_{\R^2}|x|^{-\varrho}|f_n|^{p+1}\,dx
&=&\int_{|x|<R}|x|^{-\varrho}|f_n|^{p+1}\,dx+\int_{|x|>R}|x|^{-\varrho}|f_n|^{p+1}\,dx,\\
&\leq&\||x|^{-\varrho}\|_{L^{\frac {2}{\varrho+\frac{1}{R}}}(|x|<R)}\|f_n\|_{L^{\frac{2(1+p)}{2-\varrho-\frac{1}{R}}}(|x|<R)}^{p+1}+|R|^{-\varrho}\|f_n\|_{L^{p+1}}^{p+1},\\
&\leq&C_R\|f_n\|_{L^{\frac{2(1+p)}{2-\varrho-\frac{1}{R}}}(|x|<R)}^{p+1}+C|R|^{-\varrho}.
\end{eqnarray*}
Owing to  $\frac{2(1+p)}{2-\varrho-\frac{1}{R}}>2$ and using a Rellich-Kondrachov compactness Theorem, we easily conclude the proof of Lemma \ref{compact}.
\end{proof}

We turn now to the proof of Proposition \ref{gag}.
\begin{proof}[{Proof of Proposition \ref{gag}}]
Define 
\begin{equation}
\label{Ju}
J(u):=\frac{\|u\|^{A}\|u\|_{\dot H^1_\alpha }^{B}}{\mathcal P[u]},
\end{equation}
and consider the minimization problem
\begin{equation}
\label{Min-Pb}
\frac{1}{\mathtt{k}_{opt}}=\inf_{0\neq u\in H_\alpha ^1}\,J(u).
\end{equation}
Let $(v_n)$ be a minimizing sequence  for \eqref{Min-Pb}, that is $v_n \in H^1_\alpha$ and 
$$\frac{1}{\mathtt{k}_{opt}}=\lim_nJ(v_n).$$
Pick
\begin{equation}
\label{lam-mu}
\mu_n:=\frac{\|v_n\|}{\|v_n\|_{\dot H^1_\alpha }},\quad \lambda_n:=\frac1{\|v_n\|_{\dot H^1_\alpha }},
\end{equation}
and define 
\begin{equation}
\label{psi-n}
\psi_n(x)=\lambda_n v_n(\mu_n x).
\end{equation}
One can easily verify that 
$$J(\psi_n)=J(v_n).$$
Hence 
\begin{equation}
\label{psi-nn}
\|\psi_n\|=\|\psi_n\|_{\dot H^1_\alpha }=1\quad\mbox{and}\quad \frac{1}{\mathtt{k}_{opt}}=\lim_n\,J(\psi_n).
\end{equation}
Then, up to a sub-sequence extraction and owing to \eqref{cpct}, there exists $\psi\in H^1_\alpha$ such that $\psi_n\rightharpoonup\psi$ in $H^1_\alpha$ and  $\psi_n \to \psi $ in $L^{p+1}(|x|^{-\varrho}\,dx)$. Consequently,
$$ J(\psi_n)=\frac1{\mathcal P[\psi_n]}\rightarrow\frac1{ \mathcal P[\psi]}\quad\mbox{as}\quad n\to\infty.$$
Using lower semi-continuity of the $H^1_\alpha-$norm, one gets 
$$\max\{\|\psi\|,\|\psi\|_{\dot H^1_\alpha }\}\leq1.$$
Hence, $J(\psi)< \frac{1}{\mathtt{k}_{opt}}$ if $\|\psi\|\|\psi\|_{\dot H^1_\alpha }<1$, which implies that
$$\|\psi\|=\|\psi\|_{\dot H^1_\alpha }=1.$$
It follows that $\psi_n\rightarrow\psi\quad\mbox{in}\quad H^1_\alpha $ and
$$\frac{1}{\mathtt{k}_{opt}}=J(\psi)=\frac1{ \mathcal P[\psi]}.$$
Note that the minimizer $\psi$ satisfies the Euler-Lagrange equation
$$\partial_\varepsilon J(\psi+\varepsilon\eta)_{|\varepsilon=0}=0,\quad\forall \eta\in C_0^\infty(\R^2).$$
Using $\text{div}\left(\frac{x^\perp}{|x|^2}\right)=0$, one can see that  $\psi$ solves
\begin{equation}\label{euler}
(p-1+\varrho)\mathcal K_\alpha \psi+(2-\varrho)\psi-\beta(1+p)|x|^{-\varrho}|\psi|^{p-1}\psi=0.
\end{equation}
Let us pick
$$
\lambda:=\Big((\frac{A}{B})^{-\frac\varrho2}\frac{A}{\beta(1+p)}\Big)^\frac1{p-1},\quad \mu:=\Big(\frac{A}{B}\Big)^\frac1{2},
$$
and re-scale the function $\psi$ as $\psi(x)=\lambda \phi(\mu x)$. A straightforward computation yields
$$\mathcal K_\alpha \phi+\phi-|x|^{-\varrho}|\phi|^{p-1}\phi=0.$$
Moreover, since $\|\psi\|=1=\lambda\mu^{-1}\|\phi\|$, one gets
$$
\frac{1}{\mathtt{k}_{opt}}=\frac{A}{p+1}\Big(\frac{A}{B}\Big)^{-\frac{B}2}\|\phi\|^{p-1}.
$$
This ends the proof of Proposition \ref{gag}.
\end{proof}


\section{Local Theory}
\label{S4}
Our aim in this section is to investigate the local well-posedness of \eqref{S} in the energy space ${H}^1_{\alpha}$. We have the following:
\begin{thm}
\label{loc}
Let $\alpha\in\R\setminus\Z$, $0<\varrho<1$, $p>1$ and $u_0\in H^1_\alpha(\R^2)$. Then there exists $T=T(\|u_0\|_{H^1_\alpha})>0$ and a unique solution $u$ to \eqref{S} with
\begin{equation}
\label{u-lwp}
u,\; \nabla_{\alpha}u\in C(0,T; L^2)\cap \left(\bigcap_{(q,r)\in\Gamma} L^q(0,T; L^r)\right).
\end{equation}
\end{thm}

Before getting into the proof of Theorem \ref{loc}, we need some technical lemmas. 
\begin{lem}
\label{rnu1}
Let $0<\varrho<1$ and $1<p<\infty$. There exist $2<r<\infty$, $0<\gamma<\frac{2}{\varrho}$ and $\frac{2}{p-1}<\nu<\infty$ such that
\begin{equation}
    \label{nu1}
    \frac{1}{r}=\frac{1}{2}-\frac{1}{\gamma}-\frac{1}{\nu}.
\end{equation}
\end{lem}
\begin{proof}
Since $0<\varrho<1$ there exists $\varepsilon \in (0,1)$ small enough such that
$$
\frac{1-\varrho}{2}-\frac{p-1}{2}\varepsilon>0.
$$
Choose $2<r<\infty$ satisfying 
$$
\frac{1}{r}<\frac{1-\varrho}{2}-\frac{p-1}{2}\varepsilon,
$$
and define 
$$
\nu=\frac{2}{(p-1)\varepsilon},\;\;\; \frac{1}{\gamma}=\frac{1}{2}-\frac{1}{\nu}-\frac{1}{r}.
$$
It follows that $\frac{2}{p-1}<\nu<\infty$ and 
$$
\frac{1}{\gamma}>\frac{1}{2}-\frac{1}{\nu}-\frac{1-\varrho}{2}+\frac{1}{\nu}=\frac{\varrho}{2}.
$$
This finishes the proof of Lemma \ref{rnu1}.
\end{proof}
\begin{lem}
\label{rnu2}
Let $0<\varrho<1$ and $1<p<\infty$. There exist $2<r<\infty$, $0<\gamma<\frac{2}{\varrho+1}$ and ${2p}<\nu<\infty$ such that
\begin{equation}
    \label{nu2}
    1=\frac{1}{r}+\frac{1}{\gamma}+\frac{p}{\nu}.
\end{equation}
\end{lem}
\begin{proof}
    Since $0<\varrho<1$ there exists $\varepsilon \in (0,1)$ small enough such that
$$
\frac{1-\varrho}{2}-\frac{\varepsilon}{2}>0.
$$
Choose $2<r<\infty$ satisfying 
$$
\frac{1}{r}<\frac{1-\varrho}{2}-\frac{\varepsilon}{2},
$$
and define 
$$
\nu=\frac{2p}{\varepsilon},\;\;\; \frac{1}{\gamma}=1-\frac{1}{r}-\frac{p}{\nu}.
$$
It follows that ${2p}<\nu<\infty$ and 
$$
\frac{1}{\gamma}>1-\frac{1-\varrho}{2}+\frac{\varepsilon}{2}-\frac{\varepsilon}{2}=\frac{1+\varrho}{2}.
$$
This finishes the proof of Lemma \ref{rnu2}.
\end{proof}
\begin{lem}
\label{q0r0}
Let $1<p<\infty$ and $(q_0,r_0)=\left(\frac{2(p+1)}{p-1}, p+1\right)$. Then $(q_0,r_0)\in\Gamma$ and for any $T>0$,
\begin{equation}
\label{u-v}
\||u|^{p-1}v\|_{L_T^{q_0'}(L^{r_0'})}\lesssim T^{\frac{2}{p+1}}\,\|\langle\nabla_{\alpha}\rangle u\|_{L_T^\infty(L^2)}^{p-1}\,\|v\|_{L_T^{q_0}(L^{r_0})}. 
\end{equation}
\end{lem}
\begin{proof}
We use H\"older's inequality, the Sobolev embeddin \eqref{Sob-emb} and the equality $\frac{1}{r_0'}=\frac{1}{r_0}+\frac{p-1}{r_0}$ and obtain that
\begin{eqnarray*}
\||u|^{p-1}v\|_{L_T^{q_0'}(L^{r_0'})}&\leq& \||u|^{p-1}\|_{L_T^{\frac{p+1}{2}}(L^{\frac{r_0}{p-1}})}\, \|v\|_{L_T^{q_0}(L^{r_0})},\\
&\leq& \|u\|_{L_T^{\frac{p^2-1}{2}}(L^{r_0})}^{p-1}\, \|v\|_{L_T^{q_0}(L^{r_0})},\\
&\lesssim& \|\langle\nabla_{\alpha}\rangle u\|_{L_T^{\frac{p^2-1}{2}}(L^2)}^{p-1}\,\|v\|_{L_T^{q_0}(L^{r_0})},\\
&\lesssim&T^{\frac{2}{p+1}}\,\|\langle\nabla_{\alpha}\rangle u\|_{L_T^\infty(L^2)}^{p-1}\,\|v\|_{L_T^{q_0}(L^{r_0})}. 
\end{eqnarray*}
This gives \eqref{u-v} as desired.
\end{proof}
\begin{lem}
\label{Cont}
Let $0<\varrho<1$, $1<p<\infty$, $T>0$ and $\alpha \in\R\setminus\Z$. Then there exist $a, b>0$ depending only on $p$ and $\varrho$ such that
\begin{equation}
\label{Cont1}
\||x|^{-\varrho}|u|^{p-1}v\|_{S'(T)}\lesssim \left(T^a+T^b\right)\,\|\langle\nabla_{\alpha}\rangle u\|_{S(T)}^{p-1}\,\|v\|_{S(T)},
\end{equation}
and
\begin{equation}
\label{Cont2}
\|\nabla_{\alpha}\left(|x|^{-\varrho}|u|^{p-1}u\right)\|_{S'(T)}\lesssim \left(T^a+T^b\right)\,\|\langle\nabla_{\alpha}\rangle u\|_{S(T)}^{p}.
\end{equation}
\end{lem}
\begin{proof}
We will making use of the elementary observation:
\begin{equation}
\label{Observ}
|x|^{-\varrho}\in L^{\gamma}(\bB)\;\;\;\text{if}\;\;\; \frac{2}{\gamma}>\varrho\quad\text{and}\quad |x|^{-\varrho}\in L^{\gamma}(\bB^{c})\;\;\;\text{if}\;\;\; \frac{2}{\gamma}<\varrho,
\end{equation}
where $\bB=\{x\in\R^2;\;\;|x|<1\}$ is the unit ball in $\R^2$. 

First, let us prove the estimate \eqref{Cont1}. Write
$$
\||x|^{-\varrho}|u|^{p-1}v\|_{S'(T)}\leq \bI_1+\bI_2,
$$
where
\begin{eqnarray*}
\bI_1&=&\||x|^{-\varrho}|u|^{p-1}v\|_{S'(\bB,T)},\\
\bI_2&=&\||x|^{-\varrho}|u|^{p-1}v\|_{S'(\bB^{c},T)}.
\end{eqnarray*}
Let $r, \gamma, \nu$ be as in Lemma \ref{rnu1} and $q$ such that $(q,r)\in\Gamma$. By H\"older's inequality, \eqref{Observ} and \eqref{Sob-emb} ,
\begin{eqnarray*}
\bI_1&\leq& \||x|^{-\varrho}\|_{L^{\gamma}(\bB)}\,\||u|^{p-1}\|_{L^{q'}_T(L^{\nu})}\,\|v\|_{L^\infty_T(L^2)},\\
&\lesssim& \|u\|_{L^{(p-1)q'}_T(L^{\nu(p-1)})}^{p-1}\,\|v\|_{L^\infty_T(L^2)},\\
&\lesssim &T^{\frac{1}{q'}}\, \|u\|_{L^{\infty}_T(H^{1}_\alpha)}^{p-1}\,\|v\|_{L^\infty_T(L^2)},\\
&\lesssim&T^{\frac{1}{q'}}\, \|\langle\nabla_{\alpha}\rangle u\|_{S(T)}^{p-1}\,\|v\|_{S(T)}.
\end{eqnarray*}
The estimate of $\bI_2$ easily follows from Lemma \ref{q0r0}. Indeed, let $(q_0,r_0)$ be as in Lemma \ref{q0r0}. Then
\begin{eqnarray*}
\bI_2&\lesssim& \||u|^{p-1}v\|_{L_T^{q_0'}(L^{r_0'})},\\
&\lesssim& T^{\frac{2}{p+1}}\,\|\langle\nabla_{\alpha}\rangle u\|_{L_T^\infty(L^2)}^{p-1}\,\|v\|_{L_T^{q_0}(L^{r_0})},\\
&\lesssim& T^{\frac{2}{p+1}}\,\|\langle\nabla_{\alpha}\rangle u\|_{S(T)}^{p-1}\,\|v\|_{S(T)}.
\end{eqnarray*}
This finishes the proof of \eqref{Cont1}.\\

We turn now to \eqref{Cont2}. Clearly
$$
\|\nabla_{\alpha}\left(|x|^{-\varrho}|u|^{p-1}u\right)\|_{S'(T)}\leq \bJ_1+\bJ_2,
$$
where 
\begin{eqnarray*}
\bJ_1&=&\|\nabla_{\alpha}\left(|x|^{-\varrho}|u|^{p-1}u\right)\|_{S'(\bB,T)},\\
\bJ_1&=&\|\nabla_{\alpha}\left(|x|^{-\varrho}|u|^{p-1}u\right)\|_{S'(\bB^c, T)}.
\end{eqnarray*}
Using the fact that $|\nabla_{\alpha}\,f|\lesssim |\nabla\,f|+\frac{|f|}{|x|}$, we get
\begin{eqnarray*}
\bJ_1&\lesssim& \||x|^{-\varrho}|u|^{p-1}\nabla\,u\|_{S'(\bB,T)}+\||x|^{-\varrho-1}|u|^{p-1}u\|_{S'(\bB,T)},\\
\bJ_2&\lesssim& \||x|^{-\varrho}|u|^{p-1}\nabla\,u\|_{S'(\bB^c,T)}+\||x|^{-\varrho-1}|u|^{p-1}u\|_{S'(\bB^c,T)}.
\end{eqnarray*}
Arguing as for $\bI_1$ and owing to $\|\nabla\,f\|\lesssim \|\nabla_\alpha\,f\|$, we infer that
$$
\||x|^{-\varrho}|u|^{p-1}\nabla\,u\|_{S'(\bB,T)}\lesssim T^a\, \|\langle\nabla_{\alpha}\rangle u\|_{S(T)}^{p-1}\,\|\nabla_\alpha\,u\|_{S(T)}\lesssim T^a\,\|\langle\nabla_{\alpha}\rangle u\|_{S(T)}^{p},
$$
for some positive constant $a$. Next we bound $\||x|^{-\varrho-1}|u|^{p-1}u\|_{S'(\bB,T)}$. Let $r, \gamma, \nu$ be as in Lemma \ref{rnu2} and $q$ such that $(q,r)\in\Gamma$. By H\"older's inequality, \eqref{Observ} and \eqref{Sob-emb}, we get
\begin{eqnarray*}
\||x|^{-\varrho-1}|u|^{p-1}u\|_{S'(\bB,T)}&\leq& \||x|^{-\varrho-1}\|_{L^{\gamma}(\bB)}\,\|u\|_{L^{(p-1)q'}_T(L^{\nu})}^{p-1}\,\|u\|_{L^\infty_T(L^{\nu})},\\
&\lesssim& \|\langle\nabla_{\alpha}\rangle u\|_{L^{(p-1)q'}_T(L^{2})}^{p-1}\,\|\langle\nabla\rangle u\|_{L^\infty_T(L^2)},\\
&\lesssim&T^{\frac{1}{q'}}\, \|\langle\nabla_{\alpha}\rangle u\|_{S(T)}^{p}.
\end{eqnarray*}
Therefore $\bJ_1\lesssim T^{a}\, \|\langle\nabla_{\alpha}\rangle u\|_{S(T)}^{p}$. We turn now to the term $\bJ_2$. Arguing as for $\bI_2$, we obtain that
\begin{eqnarray*}
\||x|^{-\varrho}|u|^{p-1}\nabla\,u\|_{S'(\bB^c,T)}&\lesssim& T^{\frac{2}{p+1}}\,\|\langle\nabla_{\alpha}\rangle u\|_{S(T)}^{p-1}\,\|\nabla u\|_{S(T)},\\
\||x|^{-\varrho-1}|u|^{p-1}u\|_{S'(\bB^c,T)}&\lesssim& T^{\frac{2}{p+1}}\,\|\langle\nabla_{\alpha}\rangle u\|_{S(T)}^{p-1}\,\|u\|_{S(T)}.
\end{eqnarray*}
It follows that $\bJ_2\lesssim T^{\frac{2}{p+1}}\, \|\langle\nabla_{\alpha}\rangle u\|_{S(T)}^{p}$. This finishes the proof of \eqref{Cont2}.
\end{proof}

Having at hand the above technical results, we are now able to prove Theorem \ref{loc}.
\begin{proof}[{Proof of Theorem \ref{loc}}]

Thanks to the Duhamel formula, solutions of \eqref{S} are fixed points of the integral functional 
\begin{align}
\Phi(u)(t)&:= e^{it\mathcal K_\alpha }u_0- i\int_0^t e^{i(t-s)\mathcal K_\alpha }[|x|^{-\varrho}|u(s)|^{p-1}u(s)]\,ds.\label{duhamel}
\end{align}
For $R,T>0$ to be chosen later, let
$$
\mathbf{X}(T,R)=\bigg\{\, u\in C_T(H^1_\alpha)\;\;\text{s.t.}\;\; u, \nabla_\alpha u \in \displaystyle{\bigcap_{(q,r)\in\Gamma}}L^q_T(L^r)\;\;\text{and}\;\; \|u\|_{S(T)}\leq R\,\bigg\},
$$
endowed with the distance $d(u,v)=\|u-v\|_{S(T)}$. Clearly $(\mathbf{X}(T,R), d)$ is a complete metric space. Applying Strichartz estimates \eqref{Str-est1} and \eqref{Str-est2}, we get
\begin{eqnarray*}
\|\Phi(u)\|_{S(T)}&\lesssim& \|u_0\|_{L^2}+\||x|^{-\varrho}|u|^{p-1}u\|_{S'(T)},\\
\|\nabla_\alpha \Phi(u)\|_{S(T)}&\lesssim& \|u_0\|_{\dot{H}^1_\alpha}+\|\nabla_\alpha\left(|x|^{-\varrho}|u|^{p-1}u\right)\|_{S'(T)},\\
\|\Phi(u)-\Phi(v)\|_{S(T)}&\lesssim& \||x|^{-\varrho}\left(|u|^{p-1}u-|v|^{p-1}v\right)\|_{S'(T)}.
\end{eqnarray*}
Employing Lemma \ref{Cont} yields
\begin{eqnarray}
\label{stab1}
\|\Phi(u)\|_{S(T)}&\leq& C\|u_0\|_{L^2}+C\left(T^a+T^b\right)\,\|u\|_{S(T)}^p,\\
\label{stab2}
\|\nabla_\alpha \Phi(u)\|_{S(T)}&\leq& C\|u_0\|_{\dot{H}^1_\alpha}+C\left(T^a+T^b\right)\,\|\langle\nabla_{\alpha}\rangle u\|_{S(T)}^{p},\\
\label{contrac}
\|\Phi(u)-\Phi(v)\|_{S(T)}&\leq& C\left(T^a+T^b\right)\,\left(\|u\|_{S(T)}^{p-1}+ \|v\|_{S(T)}^{p-1}\right)\, \|u-v\|_{S(T)},
\end{eqnarray}
for some positive constant $C$. This shows that 
$$\Phi\left( C_T(H^1_\alpha)\bigcap_{(q,r)\in\Gamma} L^q_T(W^{1,r}_\alpha)\right)\subset C_T(H^1_\alpha)\bigcap_{(q,r)\in\Gamma} L^q_T(W^{1,r}_\alpha).$$ 
Let $R=2C\|u_0\|_{H^1_\alpha}$. For $u,v \in \mathbf{X}(T,R)$, we have from \eqref{stab1} and \eqref{contrac},
\begin{eqnarray}
\label{stab1-1}
\|\Phi(u)\|_{S(T)}&\leq& \frac{R}{2}+C\left(T^a+T^b\right)R^p,\\
\label{Contract-1}
\|\Phi(u)-\Phi(v)\|_{S(T)}&\leq& 2C\left(T^a+T^b\right)R^{p-1} \|u-v\|_{S(T)}.
\end{eqnarray}
Choosing $T>0$ such that $2C\left(T^a+T^b\right)R^{p-1}<1$, we conclude the proof  by a classical fixed point argument.
\end{proof}



\section{Proof of Theorem \ref{t1}}
\label{S6}
Let us prepare the proof of the scattering. Here and hereafter, we denote by $B(R):=\{x\in\R^2;\;|x|\leq R\}$ the ball of $\R^2$ centered at the origin and with radius $R>0$ and $B^c(R)$ its complementary in $\R^2$. Also, for $0<R_1<R_2$, one denotes by $C(R_1,R_2):=\{x\in\R^2;\;R_1\leq |x|\leq R_2\}$ the annulus of $\R^2$.
Let $\psi\in C_0^\infty(\R^2)$ be a radial bump function such that
\begin{equation}
 \label{bump-psi} 
 \psi=1\quad\mbox{on}\quad B\left(\frac12\right),\quad\psi=0 \quad \mbox{on}\quad B^c(1)\quad\mbox{and}\quad 0\leq\psi\leq 1.
\end{equation}
For $R>0$, define
\begin{equation}
\label{psi-R}
\psi_{R}(x)=\psi\left(\frac{|x|}{R}\right).
\end{equation}
\subsection{Variational Analysis}

Recall that $\phi$ stands for a radially symmetric decreasing solution to \eqref{gr}. The following inequality will be useful in obtaining a coercivity result.
\begin{lem}
\label{Coe}
Let $u\in H^1_\alpha(\R^2)$. Then,
\begin{equation}
\label{Coer0}
\mathcal P[u]
\leq\frac{p+1}{B}\bigg(\frac{M[u]^{\lambda_c}\mathcal P[u]}{M[\phi]^{\lambda_c}\mathcal P[\phi]}\bigg)^{\frac{B-2}{B}}\;\|u\|_{\dot H^1_\alpha }^2,
\end{equation}
where $B$ is given by \eqref{AB} and $\lambda_c$ is given by \eqref{gam-c}.
\end{lem}
\begin{proof}
Thanks to Pohozaev identities, one has
\begin{equation}\label{poh}
\mathcal P[\phi]=\frac{p+1}{A}\,M[\phi]=\frac{p+1}{B}\Big(\|\nabla\phi\|^2+\alpha^2\||x|^{-1}\phi\|^2\Big).
\end{equation}
Indeed, multiplying \eqref{gr} with $\bar\phi$ and integrating, it follows that
$$\int_{\R^2}\mathcal K_\alpha(\phi)\bar\phi\,dx+\|\phi\|^2=\mathcal P[\phi].$$
Moreover, since $\text{div}\left(\frac{x^\perp}{|x|^2}\right)=0$, an integration by parts gives
\begin{eqnarray*}
\int_{\R^2}\mathcal K_\alpha(\phi)\bar\phi\,dx
&=&\int_{\R^2}\Big(-\Delta\phi-2i\frac{\alpha}{|x|^2}x^\perp\cdot\nabla\phi+\frac{\alpha^2}{|x|^2}\Big)\bar\phi\,dx\\
&=&\|\nabla\phi\|^2-2i\alpha\int_{\R^2}\frac{x^\perp}{|x|^2}\cdot\nabla\phi\,\bar\phi\,dx+\alpha^2\left\|\frac{\phi}{|x|}\right\|^2\\
&=&\|\nabla\phi\|^2+2\alpha\Re\left(i\int_{\R^2}\frac{x^\perp}{|x|^2}\cdot\overline{\nabla\phi}\phi\,dx\right)+\alpha^2\left\|\frac{\phi}{|x|}\right\|^2\\
&=&\|{\nabla_\alpha}\phi\|^2.
\end{eqnarray*}
Thus,
$$\|{\nabla_\alpha}\phi\|^2+\|\phi\|^2=\mathcal P[\phi].$$
Define the action
$$S(\phi):=\|{\nabla_\alpha}\phi\|^2+\|\phi\|^2-\frac2{p+1}\mathcal P[\phi].$$
Since $S'(\phi)=0$, one gets $\partial_\lambda\Big(S(\phi_{\alpha,\beta}^\lambda)\Big)_{\lambda=1}=0$, where $\phi_{\alpha,\beta}^\lambda(x):=\lambda^\alpha(\phi(\lambda^\beta x)$. A straightforward computation gives
\begin{gather*}
\|{\nabla_\alpha}\phi^{\lambda}_{\alpha,\beta}\|^2=\lambda^{2\alpha}\Big(\|\nabla \phi\|^2+\||x|^{-1}u\|^2\Big),\\
\|\phi^{\lambda}_{\alpha,\beta}\|=\lambda^{\alpha-\beta}\|\phi\|,\\
\mathcal P[\phi^{\lambda}_{\alpha,\beta}]=\lambda^{\alpha(1+p)+\beta(-2+\varrho)} \mathcal P[\phi].
\end{gather*}
Therefore
\begin{eqnarray*}
\partial_\lambda\Big(S(\phi_{\alpha,\beta}^\lambda)\Big)_{|\lambda=1}
&=&2\alpha\|{\nabla_\alpha}\phi\|^2+2(\alpha-\beta)\|u\|^2-2\frac{\alpha(1+p)+\beta(-2+\varrho)}{p+1}\mathcal P[\phi].
\end{eqnarray*}
By taking $\alpha=\beta=1$, we obtain that
$$\|{\nabla_\alpha}\phi\|^2=\frac{B}{p+1}\mathcal P[\phi].$$
This leads to
$$\|\phi\|^2=(1-\frac{B}{p+1})\mathcal P[\phi]=\frac{A}{p+1}\mathcal P[\phi].$$
Using the Gagliardo-Nirenberg inequality \eqref{ineq}, the expression of $\mathtt{K}_{opt}$ given by \eqref{part3}, the pohozaev identities \eqref{poh} and the identities $(p-1)i_c=B-2$ and $\lambda_c(B-2)=A$, one writes
\begin{eqnarray*}
[\mathcal P[u]]^\frac{B}2
&\leq& {\tt K}_{opt} \left(\|u\|^{2\lambda_c}\mathcal P[u]\right)^{\frac B2-1}\,\| u\|_{\dot H^1_\alpha }^B\\
&\leq& \frac{p+1}{A}\left(\frac AB\right)^{\frac{B}2}\|\phi\|^{-(p-1)}\left(M[u]^{\lambda_c}\mathcal P[u]\right)^{\frac B2-1}\| u\|_{\dot H^1_\alpha }^B\\
&\leq& \frac{p+1}A\left(\frac AB\right)^{\frac{B}2}M[\phi]^{\frac{A-(p-1)}2}[\mathcal P[\phi]]^{\frac B2-1}\bigg(\frac{M[u]^{\lambda_c}\mathcal P[u]}{M[\phi]^{\lambda_c}\mathcal P[\phi]}\bigg)^{\frac B2-1}\| u\|_{\dot H^1_\alpha }^B\\
&\leq&\bigg(\frac AB\frac{\mathcal P[\phi]}{M[\phi]}\bigg)^{\frac{B}2}\,\bigg(\frac{M[u]^{\lambda_c}\mathcal P[u]}{M[\phi]^{\lambda_c}\mathcal P[\phi]}\bigg)^{\frac B2-1}\,\| u\|_{\dot H^1_\alpha }^B\\
&\leq&\bigg(\frac{M[u]^{\lambda_c}\mathcal P[u]}{M[\phi]^{\lambda_c}\mathcal P[\phi]}\bigg)^{\frac B2-1}\,\bigg(\frac{p+1}B\| u\|_{\dot H^1_\alpha }^2\bigg)^{\frac{B}2}.
\end{eqnarray*}
This leads to \eqref{Coer0} as desired.
\end{proof}
As a consequence of the above lemma, we obtain the following coercity result.
\begin{cor}\label{bnd}
Let $u\in H_\alpha ^1(\R^2)$ and $\varepsilon\in (0,1)$ satisfying
\begin{equation}\label{1}
\mathcal P[u][M[u]]^{\lambda_c}\leq(1-\varepsilon)\mathcal P[\phi][M[\phi]]^{\lambda_c}.
\end{equation}
Then, 
\begin{equation}
\label{Coer1}
\mathcal P[u]\leq \frac{p+1}{B}\;(1-\varepsilon)^{\frac{B-2}{B}}\;\|u\|_{\dot H^1_\alpha }^2\leq \frac{p+1}{B}\;\| u\|_{\dot H^1_\alpha }^2,
\end{equation}
and 
\begin{equation}
\label{Coer2}
\| u\|_{\dot H^1_\alpha }^2-\frac{B}{p+1}\mathcal P[u]\geq c(\varepsilon, B)\;\| u\|_{\dot H^1_\alpha }^2,
\end{equation}
where $c(\varepsilon ,B):= 1-(1-\varepsilon)^{\frac{B-2}{B}}>0$ since $B>2$.
Moreover, for $\varepsilon$ small enough, we have
\begin{equation}
\label{Coer3}
E[u]\geq \frac{B-2}{B}\;\| u\|_{\dot H^1_\alpha }^2.
\end{equation}
\end{cor}

\begin{proof}
Inequality \eqref{Coer1} follows immediately from \eqref{Coer0} and \eqref{1}.
To prove \eqref{Coer2} we use the first inequality in \eqref{Coer1} and the fact that $B>2$ and $\varepsilon\in (0,1)$.
Finally, using the first inequality in \eqref{Coer1}, we infer
\begin{eqnarray*}
E[u]
&=&\| u\|_{\dot H^1_\alpha }^2-\frac2{p+1}\mathcal P[u]\\
&\geq&\bigg(1-\frac{2}B\Big(1-\varepsilon\Big)^\frac{B-2}B\bigg)\| u\|_{\dot H^1_\alpha }^2 .
\end{eqnarray*}
Since $1-\frac{2}B\Big(1-\varepsilon\Big)^\frac{B-2}B\to 1-\frac{2}{B}$ as $\varepsilon\to 0$ and $B>2$, we get \eqref{Coer3}.
\end{proof}
\begin{rem}
Since $$\mathcal P[\psi_R\,u]\leq \mathcal P[u]\quad\mbox{and}\quad M[\psi_R\,u]\leq M[u], \;\;\; \forall\;\;\;R>0,$$
inequalities \eqref{Coer1}-\eqref{Coer2} remain true for $\psi_R\,u$ instead of $u$. Namely, we have
\begin{equation}
\label{Coer1-R}
\mathcal P[\psi_R\,u]\leq \frac{p+1}{B}\;\|\psi_Ru\|_{\dot H^1_\alpha }^2,
\end{equation}
and
\begin{equation}
\label{Coer2-R}
\|\psi_Ru\|_{\dot H^1_\alpha }^2-\frac{B}{p+1}\mathcal P[\psi_R\,u]\geq c(\varepsilon, B)\;\|\psi_Ru\|_{\dot H^1_\alpha }^2.
\end{equation}
\end{rem}
\begin{rem}
The solution is global by \eqref{Coer3}.
\end{rem}

\subsection{Morawetz estimate}
Let $b=b(x)=b(|x|):\R^2\to\R$ be a radial smooth function sufficiently decaying at infinity and $u\in C([0,T^*); H^1_{\alpha})$ be the maximal solution of \eqref{S}. One denotes the virial potential
\begin{equation}
\label{Vb}
V_b:t\mapsto\int_{\R^2}b(x)|u(t,x)|^2\,dx.
\end{equation}
Here and hereafter, subscripts denote partial derivatives and repeated indexes are summed and $\partial_r$ is the radial derivative. 
Now, taking into account \cite[Theorem 1.2, p. 252]{fv} and \cite[Theorem 3.1, p. 9]{Garcia} for $V:=-|x|^{-\varrho}|u|^{p-1}$, one has
\begin{eqnarray*}
V_b''(t)
&=&-\int_{\R^2}\Delta^2b|u|^2\,dx+4\int_{\R^2}\nabla_\alpha uD^2b\overline{\nabla_\alpha u}\,dx\\
&+&4\Im\Big(\int_{\R^2}ub'{\bf B}\frac{x}{|x|}\cdot\overline{\nabla_\alpha u}\,dx\Big)-2\int_{\R^2}\left(\nabla b\cdot \nabla V\right)|u|^2\,dx.
\end{eqnarray*}
Here,
$${\bf B}(x)\in \mathcal{M}_{2\times 2}(\R),\quad {\bf B}_{ij}:=\partial_{x_j}\bA_\alpha^i-\partial_{x_i}\bA_\alpha^j,$$
where $\bA_\alpha$ is given by \eqref{A-alpha}.
A direct calculus gives ${\bf B}=0$ and so with integration by parts
\begin{eqnarray*}
V_b''(t)
&=&-\int_{\R^2}\Delta^2b|u|^2\,dx+4\int_{\R^2}\nabla_\alpha uD^2b\overline{\nabla_\alpha u}\,dx\\
&+&2\int_{\R^2}\nabla b\cdot\nabla(|x|^{-\varrho}|u|^{p-1})|u|^2\,dx\\
&=&-\int_{\R^2}\Delta^2b|u|^2\,dx+4\int_{\R^2}\nabla_\alpha uD^2b\overline{\nabla_\alpha u}\,dx\\
&-&2\int_{\R^2}\Delta b|x|^{-\varrho}|u|^{1+p}\,dx-2\int_{\R^2}\nabla b\cdot\nabla(|u|^2)|x|^{-\varrho}|u|^{p-1}\,dx.
\end{eqnarray*}

Also, with integration by parts
\begin{eqnarray}
V_b''(t)
&=&-\int_{\R^2}\Delta^2b|u|^2\,dx+4\int_{\R^2}\nabla_\alpha uD^2b\overline{\nabla_\alpha u}\,dx\nonumber\\
&-&2\int_{\R^2}\Delta b|x|^{-\varrho}|u|^{1+p}\,dx-\frac4{1+p}\int_{\R^2}\nabla b\cdot\nabla(|u|^{1+p})|x|^{-\varrho}\,dx\nonumber\\
&=&-\int_{\R^2}\Delta^2b|u|^2\,dx+4\int_{\R^2}\nabla_\alpha uD^2b\overline{\nabla_\alpha u}\,dx\nonumber\\
&-&\frac{2(p-1)}{1+p}\int_{\R^2}\Delta b|x|^{-\varrho}|u|^{1+p}\,dx+\frac4{1+p}\int_{\R^2}\nabla b\cdot\nabla(|x|^{-\varrho})|u|^{1+p}\,dx.\label{V-b}
\end{eqnarray}

Consider, for $R>0$, a smooth radial real-valued function $f:=f_R$ such that
\begin{equation}
\label{f}
f(r):=\left\{
\begin{array}{ll}
{r^2},\quad\mbox{if}\quad 0\leq r\leq R,\\
3Rr,\quad\mbox{if}\quad  r>2R,
\end{array}
\right.
\end{equation}
such that on the annulus $C(R,2R)$ one has
\begin{equation}\label{ff}
\min\{\partial_rf,\partial^2_rf\}\geq0,\quad |\partial^\gamma f|\lesssim R|x|^{1-|\gamma|}.
\end{equation}

Under these conditions, the matrix $(f_{jk})$ is non-negative. Moreover, by the radial identity
\begin{equation}\label{symm}
\frac{\partial^2}{\partial x_l\partial x_k}:=\partial_l\partial_k=\Big(\frac{\delta_{lk}}r-\frac{x_lx_k}{r^3}\Big)\partial_r+\frac{x_lx_k}{r^2}\partial_r^2,
\end{equation}
one gets
$$
\left\{
\begin{array}{ll}
f_{jk}=2\delta_j^k,\quad\Delta f=4,\quad \Delta^2f=0,\quad 0\leq r\leq R,\\
f_{jk}=\frac{3R}r[\delta_j^k-\frac{x_jx_k}{r^2}],\quad\Delta f=\frac{3R}r,\quad \Delta^2f=\frac{3R}{r^3},\quad  r>2R.
\end{array}
\right.
$$
Denote by $V_R:=V_{f_R}$ and $\not\nabla_\alpha :=\nabla_\alpha -\frac{x\cdot\nabla_\alpha}{|x|^2}x$ the angular gradient. Thus,
\begin{align*}
 V_R''(t)
&=8\int_{B(R)}\Big(|\nabla_\alpha u|^2-\frac{B}{p+1}|x|^{-\varrho}|u|^{p+1}\Big)\,dx\\
&+\int_{B^c(2R)}\Big(\frac{12R}{|x|}|\not\!\nabla_\alpha u|^2-3R\frac{|u|^2}{|x|^3}-\frac{6R(p-1+2\varrho)}{p+1}|x|^{-\varrho-1}|u|^{p+1}\Big)\,dx\\
&+\int_{C(R,2R)}\Big(-\Delta^2f|u|^2+4\nabla_\alpha uD^2f\overline{\nabla_\alpha u}-\frac{2(p-1)}{p+1}\Delta f|x|^{-\varrho}|u|^{p+1}\Big)dx\\&+\int_{C(R,2R)}\Big(\frac{4}{p+1}\nabla f\cdot\nabla(|x|^{-\varrho})|u|^{p+1}\Big)\,dx.
\end{align*}
It follows that
\begin{align*}
V_R''(t)&\geq 8\int_{B(R)}\Big(|\nabla_\alpha u|^2-\frac{B}{p+1}|x|^{-\varrho}|u|^{p+1}\Big)\,dx\\&-\int_{B^c(2R)}\Big(3R\frac{|u|^2}{|x|^3}+\frac{6R(p-1+2\varrho)}{p+1}|x|^{-\varrho-1}|u|^{p+1}\Big)\,dx\\
&+\int_{C(R,2R)}\Big(-\Delta^2f|u|^2+4\nabla_\alpha uD^2f\overline{\nabla_\alpha u}-\frac{2(p-1)}{p+1}\Delta f|x|^{-\varrho}|u|^{p+1}\Big)dx\\&+\int_{C(R,2R)}\Big(\frac{4}{p+1}\nabla f\cdot\nabla(|x|^{-\varrho})|u|^{p+1}\Big)\,dx.
\end{align*}

Moreover, by \eqref{symm} and \eqref{ff}, one writes

\begin{eqnarray*}
\int_{C(R,2R)}\nabla_\alpha uD^2f\overline{\nabla_\alpha u}\,dx
&=&\int_{C(R,2R)}(\nabla_\alpha u)_l\Big[\Big(\frac{\delta_{lk}}r-\frac{x_lx_k}{r^3}\Big)f'+\frac{x_lx_k}{r^2}f''\Big](\overline{\nabla_\alpha u})_k\,dx\\
&=&\int_{C(R,2R)}\Big[|\not\!\nabla_\alpha u|^2\frac{f'}{|x|}+|x\cdot\nabla_\alpha u|^2\frac{f''}{|x|^2}\Big]\,dx\\
&\geq&0.
\end{eqnarray*}

Hence, by \eqref{ff} and Sobolev embeddings, one gets
\begin{align*}
V_R''(t)
&\gtrsim\int_{B(R)}\Big(|\nabla_\alpha u|^2-\frac{B}{p+1}|x|^{-\varrho}|u|^{p+1}\Big)\,dx-\frac{\|u\|^2}{R^2}-\frac{\|u\|_{H^1}^{p+1}}{R^{\varrho}}.
\end{align*}
Owing to \eqref{Coer3} and \eqref{Coer2-R}, we infer that
\begin{align}
\frac{1}{R^2}+\frac{1}{R^{\varrho}}+V_R''(t)
&\gtrsim\int_{B(R)}\Big(|\nabla_\alpha u|^2-\frac{B}{p+1}|x|^{-\varrho}|u|^{p+1}\Big)\,dx\nonumber\\
&\gtrsim\|\psi_Ru\|_{\dot H^1_\alpha }^2\nonumber\\
&\gtrsim\int_{\R^2}|x|^{-\varrho}|\psi_Ru|^{p+1}\,dx\nonumber\\
&\gtrsim\int_{B(R)}|x|^{-\varrho}|u|^{p+1}\,dx.\label{v''}
\end{align}
As a consequence, we have:
\begin{prop}\label{dc2}
Let $T>0$, $0<\varrho<1$ and $u\in C_T(H^1_\alpha)$ be a solution of \eqref{S}. Then
\begin{equation}
\label{Mor-est}
\int_0^T\,\int_{\R^2}|x|^{-\varrho}|u(t,x)|^{p+1}\,dx\,dt\lesssim T^{\frac{1}{1+\varrho}}.
\end{equation}

\end{prop}
\begin{proof}
By \eqref{v''} together with Sobolev embedding, one has
\begin{eqnarray*}
\int_{\R^2}|x|^{-\varrho}|u(t,x)|^{p+1}\,dx
&=&\int_{B(R)}|x|^{-\varrho}|u(t,x)|^{p+1}\,dx+\int_{B(R)^c}|x|^{-\varrho}|u(t,x)|^{p+1}\,dx\\
&\lesssim&\frac{1}{R^2}+\frac{1}{R^{\varrho}}+V_R''(t)+\frac{1}{R^{\varrho}}\int_{\R^2}|u(t,x)|^{p+1}\,dx\\
&\lesssim&\frac{1}{R^2}+\frac{1}{R^{\varrho}}+V_R''(t).
\end{eqnarray*}
This gives
\begin{eqnarray*}
\int_0^T\,\int_{\R^2}|x|^{-\varrho}|u(t,x)|^{p+1}\,dx\,dt
&\lesssim&\frac{T}{R^2}+\frac{T}{R^{\varrho}}+V_R'(T)-V_R'(0)\\
&\lesssim&\frac{T}{R^{\varrho}}+R.
\end{eqnarray*}
Taking $R=T^{\frac{1}{1+\varrho}}$, yields \eqref{Mor-est}.
\end{proof}
From \eqref{Mor-est}, one ca see that there exist 
 $t_n,R_n\to\infty$ such that 
 \begin{equation}
\label{tn-Rn}
\lim_{n\to\infty}\int_{B(R_n)}|x|^{-\varrho}|u(t_n,x)|^{p+1}\,dx=0.
 \end{equation}
 Indeed, using \eqref{Mor-est}, one has
$$\frac2T\int_{T/2}^T\int_{\R^2}|x|^{-\varrho}|u(t,x)|^{p+1}\,dx\,dt\lesssim T^{-\frac{\varrho}{1+\varrho}}.$$ We conclude the proof by using the mean value Theorem.

\subsection{Scattering Criterion}

In this section we give a scattering criterion as stated below. 
\begin{prop}\label{crt}
Under the same assumptions as in Theorem \ref{t1}, let $u\in C(\R,H_\alpha ^1)$ be a global solution to \eqref{S} satisfying 
\begin{equation}
\label{Bound-s}
0<\sup_{t\geq0}\|u(t)\|_{H^1_\alpha }:=E<\infty.
\end{equation}

Then, there exist $R,\varepsilon>0$ depending on $E,p,\varrho$ such that if
\begin{equation}\label{crtr} 
\liminf_{t\to\infty}\int_{|x|<R}|u(t,x)|^2\,dx<\varepsilon^2,
\end{equation}
then, $u$ scatters for positive time. 
\end{prop}

The key of the proof of the scattering criterion is the next result.
\begin{prop}\label{fn}
Let the assumptions of Proposition \ref{crt} be fulfilled. Then, for any $\varepsilon>0$, there exist $T,\mu>0$ satisfying 
$$\|e^{i(t-T){\mathcal K_\alpha }}u(T)\|_{L^\frac{8(p-1)}{3(2-\varrho)}((T,\infty),L^{\frac{8(p-1)}{2-\varrho}})}\lesssim \varepsilon^\mu.$$
\end{prop}
\begin{rem}
Note that $\Big(\frac{8(p-1)}{3(2-\varrho)},\frac{8(p-1)}{2-\varrho}\Big)=\Big(\frac{8}{3(1-s_c)},\frac{8}{1-s_c}\Big)$ where $s_c=1-\frac{2-\varrho}{p-1}$ is the critical Sobolev index.
\end{rem}
\begin{proof}[{Proof of Proposition \ref{fn}}]
Take $0<\varepsilon<<1$, $R(\varepsilon)>>1$ to be fixed later and $I=[0,T]$ a time slab. By the integral formula \eqref{duhamel},
\begin{eqnarray*}
e^{i(t-T){\mathcal K_\alpha }}u(T)
&=&e^{it{\mathcal K_\alpha }}u_0+i\int_0^Te^{i(t-\tau){\mathcal K_\alpha }}[|x|^{-\varrho}|u|^{p-1}u]\,d\tau\\
&=&e^{it{\mathcal K_\alpha }}u_0+i\Big(\int_0^{T-\varepsilon^{-\beta}}+\int_{T-\varepsilon^{-\beta}}^T\Big)e^{i(t-\tau){\mathcal K_\alpha }}[|x|^{-\varrho}|u|^{p-1}u]\,d\tau\\
&:=&e^{it{\mathcal K_\alpha }}u_0+i\Big(\int_{I_1}+\int_{I_2}\Big)e^{i(t-\tau){\mathcal K_\alpha }}[|x|^{-\varrho}|u|^{p-1}u]\,d\tau\\
&:=&e^{it{\mathcal K_\alpha }}u_0+F_1+F_2.
\end{eqnarray*}
\begin{itemize}
    \item[$\bullet$] {\bf The linear term}.\newline Employing the dispersive estimate \eqref{Dis-est} together with Sobolev embedding and $p>p_c=3-\varrho$, we get
\begin{eqnarray*}
\|e^{it{\mathcal K_\alpha }}u_0\|_{L^\frac{8(p-1)}{3(2-\varrho)}((T,\infty),L^{\frac{8(p-1)}{2-\varrho}})}
&\lesssim&\|\frac1{t^{1-\frac{2-\varrho}{4(p-1)}}}\|_{L^\frac{8(p-1)}{3(2-\varrho)}(T,\infty)}\|u_0\|_{L^\infty(H^1)}\\
&\lesssim&\left(\int_T^\infty\frac{dt}{t^{\frac23[\frac{4(p-1)}{2-\varrho}-1]}}\right)^\frac{3(2-\varrho)}{8(p-1)}.
\end{eqnarray*}
Note that $t\longmapsto t^{\frac{2}{3}[1-\frac{4(p-1)}{2-\varrho}]}\in L^1(1,\infty)$ since $p>3-\varrho>1+\frac{5}{8}(2-\varrho)$. 
Thus, one may choose $T_0>\varepsilon^{-\beta}>0$, where $\beta>0$, such that 
$$
\|e^{it{\mathcal K_\alpha }}u_0\|_{L^\frac{8(p-1)}{3(2-\varrho)}((T_0,\infty),L^{\frac{8(p-1)}{2-\varrho}})}\leq \varepsilon^2
$$
\item[$\bullet$] {\bf The term $F_2$}.\newline  By the assumption \eqref{crtr}, one has for $T>\varepsilon^{-\beta}$ large enough,
$$\int_{\R^2}\psi_R(x)|u(T,x)|^2\,dx<\varepsilon^2.$$
Moreover, a computation with the use of \eqref{S} gives
\begin{eqnarray*}
\left|\frac d{dt}\int_{\R^2}\psi_R(x)|u(t,x)|^2\,dx\right|
&=&2\left|\int_{\R^2}\psi_R(x)\Re[\dot u(t,x)\bar u(t,x)]\,dx\right|\\
&=&2\left|\int_{\R^2}\psi_R(x)\Im[{\Delta}u(t,x)\bar u(t,x)]\,dx\right|\\
&=&2\left|\int_{\R^2}\nabla\psi_R(x)\Im[{\nabla}u(t,x)\bar u(t,x)]\,dx\right|\\
&\lesssim&\frac1{\sqrt R}.
\end{eqnarray*}
Then, for any $T-\varepsilon^{-\beta}\leq t\leq T$ and $R>\varepsilon^{-2(2+\beta)}$, one gets
$$\|\psi_Ru(t)\|\leq\left( \int_{\R^2}\psi_R(x)|u(T,x)|^2\,dx+C\frac{T-t}{\sqrt R}\right)^\frac12\leq C\varepsilon.$$
This gives
$$\|\psi_Ru\|_{L^\infty([T-\varepsilon^{-\beta},T],L^2)}\leq C\varepsilon.$$
Moreover, by Fatou's lemma, 
$$\|u\|_{L^\infty([T-\varepsilon^{-\beta},T],L^2)}\leq\liminf_{R\to\infty}\|\psi_Ru\|_{L^\infty([T-\varepsilon^{-\beta},T],L^2)}\leq C\varepsilon.$$

Thus, by H\"older's inequality and Strichartz estimate \eqref{Loc-Str}, one writes for $\mathcal N:=|x|^{-\varrho}|u|^{p-1}u$ and $(q,r)=\left(2(1+\varepsilon),2\frac{1+\varepsilon}\varepsilon\right)\in\Gamma$, 
{
\begin{eqnarray*}
\|F_2\|_{L^\frac{8(p-1)}{3(2-\varrho)}((T,\infty),L^{\frac{8(p-1)}{2-\varrho}})}
&\lesssim&\|F_2\|_{L^{\frac{8}{3}}((T,\infty),L^8)}^{\frac{2-\varrho}{p-1}}\|F_2\|_{L^{\infty}((T,\infty),H^1_\alpha)}^{\frac{p+\varrho-3}{p-1}}\\
&\lesssim&\|\mathcal N\|_{L^{q'}(I_2,L^{r'})}+\|\nabla\mathcal N\|_{L^{q'}(I_2,L^{r'})}+\||x|^{-1}\mathcal N\|_{L^{q'}(I_2,L^{r'})}\\
&\lesssim&(I)+(II)+(III).
\end{eqnarray*}
}
For the first term, we have
\begin{eqnarray*}
(I)
&=&\||x|^{-\varrho}|u|^{p-1}u\|_{L^{q'}(I_2,L^{r'})}\\
&\leq&\|u\|_{L^\infty(I_2,L^2)}\||x|^{-\varrho}|u|^{p-1}\|_{L^{q'}(I_2,L^{1/(\frac12-\frac1r)})}\\
&\leq&\|u\|_{L^\infty(I_2,L^2)}\Big(\||x|^{-\varrho}\|_{L^{a_1}(\bB)}\|\|u\|^{p-1}_{L^{{b_1}}}\|_{L^{q'}(I_2)}+\||x|^{-\varrho}\|_{L^{a_2}(\bB^c)}\|\|u\|^{p-1}_{L^{{b_2}}}\|_{L^{q'}(I_2)}\Big)\\
&\leq&c\varepsilon\Big(\|u\|_{L^{(p-1)q'}(I_2,L^{b_1})}^{p-1}+\|u\|_{L^{(p-1)q'}(I_2,L^{b_2})}^{p-1}\Big).
\end{eqnarray*}
Here
$$
\left\{
\begin{array}{ll}
\frac12-\frac1r=\frac1{2(1+\varepsilon)}=\frac1{a_1}+\frac{p-1}{b_1}=\frac1{a_2}+\frac{p-1}{b_2};\\
a_1<\frac 2\varrho< a_2.
\end{array}
\right.
$$
Hence,
$$\frac1{2(1+\varepsilon)}-\frac{p-1}{b_1}=\frac1{a_1}>\frac\varrho 2>\frac1{a_2}=\frac1{2(1+\varepsilon)}-\frac{p-1}{b_2},$$
which is possible if 
$$0<\frac1{b_1}<\frac1{p-1}\Big(\frac1{2(1+\varepsilon)}-\frac\varrho 2\Big)<\frac1{b_2}\leq\frac12.$$
The above condition translate to $p>2-\varrho$ which is trivially satisfied when ${0<\varrho<1}$.

Since $p>p_c=3-\varrho$, we have by Sobolev embedding
\begin{eqnarray*}
(I)
&\leq&c\varepsilon\Big(\|u\|_{L^{(p-1)q'}(I_2,L^{b_1})}^{p-1}+\|u\|_{L^{(p-1)q'}(I_2,L^{b_2})}^{p-1}\Big)\\
&\leq&c\varepsilon^{1-\frac{\beta}{q'(p-1)}}\Big(\|u\|_{L^\infty(I_2,L^{(\frac{2(p-1)(1+\varepsilon)}{1-\varrho(1+\varepsilon)})^+})}^{p-1}+\|u\|_{L^\infty(I_2,L^{(\frac{2(p-1)(1+\varepsilon)}{1-\varrho(1+\varepsilon)})^-})}^{p-1}\Big)\\
&\leq&c\varepsilon^{1-\frac{\beta}{q'(p-1)}}\Big(\|u\|_{L^\infty(I_2,L^2)}^{\lambda^+(p-1)}\|u\|_{L^\infty(I_2,L^{A^+})}^{(1-\lambda^+)(p-1)}+\|u\|_{L^\infty(I_2,L^2)}^{\lambda^-(p-1)}\|u\|_{L^\infty(I_2,L^{A^-})}^{(1-\lambda^-)(p-1)}\Big)\\
&\leq&c\varepsilon^{1-\frac{2\beta(1+\varepsilon)}{(p-1)(1+2\varepsilon)}+(\frac{1-\varrho(1+\varepsilon)}{(p-1)(1+\varepsilon)})^-}.
\end{eqnarray*}
Here, one uses an interpolations $L^{(\frac{2(p-1)(1+\varepsilon)}{1-\varrho(1+\varepsilon)})^\pm}\hookrightarrow L^2\cap L^{A^\pm}$, with the Sobolev injection $L^{A^\pm}\hookrightarrow H^1$. Moreover, 
$\lambda^\pm\to (\frac{1-\varrho(1+\varepsilon)}{(p-1)(1+\varepsilon)})^\mp$ when $A^\pm\to\infty$. For the second term
\begin{eqnarray*}
(II)
&=&\|\nabla\Big(|x|^{-\varrho}|u|^{p-1}u\Big)\|_{L^{q'}(I_2,L^{r'})}\\
&\lesssim&\||x|^{-\varrho}|u|^{p-1}\nabla u\|_{L^{q'}(I_2,L^{r'})}+\||x|^{-(1+\varrho)}|u|^{p-1} u\|_{L^{q'}(I_2,L^{r'})}\\
&:=&(II)_1+(II)_2.
\end{eqnarray*}
Moreover, using the previous calculus, we get
\begin{eqnarray*}
(II)_1
&=&\||x|^{-\varrho}|u|^{p-1}\nabla u\|_{L^{q'}(I_2,L^{r'})}\\
&\leq&\|\nabla u\|_{L^\infty(I_2,L^2)}\||x|^{-\varrho}|u|^{p-1}\|_{L^{q'}(I_2,L^{1/(\frac12-\frac1r)})}\\
&\leq&c\varepsilon^{-\frac{2\beta(1+\varepsilon)}{(p-1)(1+2\varepsilon)}+(\frac{1-\varrho(1+\varepsilon)}{(p-1)(1+\varepsilon)})^-}.
\end{eqnarray*}
In the same manner as above, we get 
\begin{eqnarray*}
(II)_2
&=&\||x|^{-(1+\varrho)}|u|^{p-1}u\|_{L^{q'}(I_2,L^{r'})}\\
&\leq&\||x|^{-(1+\varrho)}\|_{L^{c_1}(\bB)}\|\|u\|^{p}_{L^{{d_1}}}\|_{L^{q'}(I_2)}+\||x|^{-(1+\varrho)}\|_{L^{c_2}(\bB^c)}\|\|u\|^{p}_{L^{{d_2}}}\|_{L^{q'}(I_2)}\Big)\\
&\leq&c\Big(\|u\|_{L^{pq'}(I_2,L^{d_1})}^{p}+\|u\|_{L^{pq'}(I_2,L^{d_2})}^{p}\Big).
\end{eqnarray*}
Here
$$
\left\{
\begin{array}{ll}
1-\frac1r=1-\frac\varepsilon{2(1+\varepsilon)}=\frac1{c_1}+\frac{p}{d_1}=\frac1{c_2}+\frac{p}{d_2};\\
c_1<\frac N\varrho< c_2.
\end{array}
\right.
$$
Thus,
$$1-\frac\varepsilon{2(1+\varepsilon)}-\frac{p}{d_1}=\frac1{c_1}>\frac\varrho N>\frac1{c_2}=1-\frac\varepsilon{2(1+\varepsilon)}-\frac{p}{d_2},$$
which is possible if
$$0<\frac1{d_1}<\frac1{p}\Big(1-\frac\varepsilon{2(1+\varepsilon)}-\frac\varrho N\Big)<\frac1{d_2}\leq\frac12.$$
The above condition can be  written as
$p>2-\varrho$ which is obviously satisfied since ${0<\varrho<1}$.
Therefore,
\begin{eqnarray*}
(II)_2
&\leq&c\Big(\|u\|_{L^{pq'}(I_2,L^{(\frac{p}{1-\frac\varepsilon{2(1+\varepsilon)}-\frac\varrho N)})^+})}^{p}+\|u\|_{L^{pq'}(I_2,L^{(\frac{p}{1-\frac\varepsilon{2(1+\varepsilon)}-\frac\varrho N)})^-})}^{p}\Big)\\
&\leq&c\varepsilon^{-\frac{\beta}{q'p}}\Big(\|u\|_{L^{\infty}(I_2,L^{(\frac{p}{1-\frac\varepsilon{2(1+\varepsilon)}-\frac\varrho N)})^+})}^{p}+\|u\|_{L^{\infty}(I_2,L^{(\frac{p}{1-\frac\varepsilon{2(1+\varepsilon)}-\frac\varrho N)})^-})}^{p}\Big)\\
&\leq&c\varepsilon^{-\frac{\beta}{q'p}}\Big(\|u\|_{L^{\infty}(I_2,L^{2})}^{p\mu^+}\|u\|_{L^{\infty}(I_2,L^{C^+})}^{p\mu^+}+\|u\|_{L^{\infty}(I_2,L^{2})}^{p\mu^-}\|u\|_{L^{\infty}(I_2,L^{C^-})}^{p\mu^-}\Big)\\
&\leq&c\varepsilon^{-\frac{\beta}{q'p}}\Big(\|u\|_{L^{\infty}(I_2,L^{2})}^{p\mu^+}+\|u\|_{L^{\infty}(I_2,L^{2})}^{p\mu^-}\Big)\\
&\leq&c\varepsilon^{-\frac{\beta}{q'p}+p\mu^-}\\
&\leq&c\varepsilon^{-\frac{\beta}{p}(1-\frac1{2(1+\varepsilon)})+(2-\frac\varepsilon{1+\varepsilon}-\varrho)^-}.
\end{eqnarray*}
The last term is estimated similarly as $(II)_2$. Regrouping the above estimates, one can see that there is $\gamma>0$ such that
\begin{eqnarray*}
&&\|F_2\|_{L^\frac{8(p-1)}{3(2-\varrho)}((T,\infty),L^{\frac{8(p-1)}{2-\varrho}})}\\
&\lesssim&\Big((I)+(II)+(III)\Big)^{\frac{2-\varrho}{2(p-1)}}\\
&\lesssim&\Big(\varepsilon^{1-\frac{2\beta(1+\varepsilon)}{(p-1)(1+2\varepsilon)}+(\frac{1-\varrho(1+\varepsilon)}{(p-1)(1+\varepsilon)})^-}+\varepsilon^{-\frac{2\beta(1+\varepsilon)}{(p-1)(1+2\varepsilon)}+(\frac{1-\varrho(1+\varepsilon)}{(p-1)(1+\varepsilon)})^-}+\varepsilon^{-\frac{\beta}{p}(1-\frac1{2(1+\varepsilon)})+(2-\frac\varepsilon{1+\varepsilon}-\varrho)^-}\Big)^{\frac{2-\varrho}{2(p-1)}}\\
&\lesssim&\varepsilon^\gamma.
\end{eqnarray*}

\item[$\bullet$] {\bf The term $F_1$}.\newline Using the dispersive estimate \eqref{Dis-est} and H\"older's inequality, one writes
\begin{eqnarray*}
\|F_1\|_\infty
&\lesssim&\int_{I_1}|t-s|^{-1}\||x|^{-\varrho}|u|^{p-1}u\|_1\,ds\\
&\lesssim&\int_{I_1}|t-s|^{-1}\Big(\||x|^{-\varrho}|u|^{p-1}u\|_{L^1(\bB)}+\||x|^{-\varrho}|u|^{p-1}u\|_{L^1(\bB^c)}\Big)\,ds\\
&\lesssim&\int_{I_1}|t-s|^{-1}\Big(\|[|x|^{-\frac\varrho{p+1}}|u|]^{p}|x|^{-\frac\varrho{p+1}}\|_{L^1(\bB)}+\|u\|_{H^1}^{p}\Big)\,ds.
\end{eqnarray*}
By H\"older's inequality and the fact that $0<\varrho<2$, one has
\begin{eqnarray*}
\|[|x|^{-\frac\varrho{p+1}}|u|]^{p}|x|^{-\frac\varrho{p+1}}\|_{L^1(\bB)}
&\leq&\||x|^{-\frac\varrho{p+1}}\|_{L^{p+1}(\bB)}\|[|x|^{-\frac\varrho{p+1}}|u|]^{p}\|_{\frac{p+1}p}\\
&\leq&\||x|^{-\varrho}\|_{L^{1}(\bB)}^\frac1{p+1}\||x|^{-\varrho}|u|^{p+1}\|_1^{\frac p{p+1}}\\
&\lesssim&\||x|^{-\varrho}|u|^{p+1}\|_1^{\frac p{p+1}}.
\end{eqnarray*}
Hence, using \eqref{ineq} and \eqref{Bound-s}, we infer that
\begin{eqnarray*}
\|F_1(t)\|_\infty
&\lesssim&\int_{I_1}|t-s|^{-1}\|u\|_{H^1_\alpha}^{p}\,ds\\
&\lesssim&E^p\,\int_{I_1}|t-s|^{-1}\,ds\\
&\lesssim& E^p\,T\,\varepsilon^\beta.
\end{eqnarray*}

Now, using an interpolation inequality via Strichartz estimates, one gets
\begin{eqnarray*}
\|F_1\|_{L^\frac{8(p-1)}{3(2-\varrho)}((T,\infty),L^{\frac{8(p-1)}{2-\varrho}})}
&\leq&\|F_1\|_{L^\frac83((T,\infty),L^8)}^\frac{2-\varrho}{p-1}\|F_1\|_{L^\infty((T,\infty),L^\infty)}^{1-\frac{2-\varrho}{p-1}}\\
&\lesssim&\|e^{i(\cdot-(T-\varepsilon^{-\beta}))\mathcal K_\alpha }u(T-\varepsilon^{-\beta})-e^{i\cdot\mathcal K_\alpha }u_0\|_{L^\frac83((T,\infty),L^8)}^\frac{2-\varrho}{p-1}
\Big(E^p T\varepsilon^\beta\Big)^{\frac{p-p_c}{p-1}}\\
&\lesssim&\Big(E^p T\varepsilon^\beta\Big)^{\frac{p-p_c}{p-1}},
\end{eqnarray*}
where $p_c=3-\varrho$.
The proof of Proposition \ref{fn} is ended if one picks $0<T<\varepsilon^{-\frac{\beta}{2}}$.
\end{itemize}
\end{proof}
Now, one can easily proves Proposition \ref{crt}. First, $u\in {L^\frac{8(p-1)}{3(2-\varrho)}((T,\infty),L^{\frac{8(p-1)}{2-\varrho}})}\cap L^\infty(\R,H^1_\alpha)$. By interpolation together with the local theory, one gets
$$u\in L^\frac{8(p-1)}{3(2-\varrho)}\left(\R;\;\;L^{\frac{8(p-1)}{4(p-1)-3(2-\varrho)}}\right), \quad \left(\frac{8(p-1)}{3(2-\varrho)},\frac{8(p-1)}{4(p-1)-3(2-\varrho)}\right)\in\Gamma.$$
The scattering follows with standard arguments.
\subsection{Proof of the scattering in Theorem \ref{t1}}
Take $R,\varepsilon>0$ given by Proposition \ref{crt} and $t_n,R_n\to\infty$ given by \eqref{tn-Rn}. Letting $n>>1$ such that $R_n>R$, one gets by H\"older's inequality
\begin{eqnarray*}
\int_{|x|\leq R}|u(t_n,x)|^2\,dx
&=&R^{\frac{2\varrho}{p+1}}\int_{|x|\leq R}|x|^{-\frac{2\varrho}{p+1}}|u(t_n,x)|^2\,dx\\
&\leq&R^{\frac{2\varrho}{p+1}}|B(R)|^{\frac{p-1}{p+1}}\,\,\||x|^{-\frac{2\varrho}{p+1}}|u(t_n,x)|^2\|_{L^\frac{p+1}2(|x|\leq R_n)}\\
&\leq&R^{\frac{2(\varrho+p-1)}{p+1}}\left(\int_{|x|\leq R_n}|x|^{-\varrho}|u(t_n,x)|^{p+1}\,dx\right)^\frac2{p+1}\\
&\lesssim& \varepsilon^2.
\end{eqnarray*}
Hence, the scattering of energy global solutions to the focusing problem \eqref{S} follows from Proposition \ref{crt}.
\subsection{Proof of the blow-up in Theorem \ref{t1}}
Let us pick $b_R(x):=R^2b(\frac{r}{R})$, $R>0$, where $b\in C_0^\infty([0,\infty))$  satisfies
$$b:r\mapsto\left\{
\begin{array}{ll}
\frac{r^2}{2},\quad\mbox{if}\quad r\leq 1,\\
0,\quad\mbox{if}\quad r\geq2,
\end{array}
\right.\quad\mbox{and}\quad b''\leq1.$$
It follows that
$$b_R''\leq1,\quad b_R'(r)\leq r\quad\mbox{and}\quad \Delta b_R\leq 2.$$

Using the spherically symmetric property \eqref{symm}, by Cauchy-Schwarz's inequality via the properties of $b$, one gets
\begin{eqnarray*}
\int_{\R^2}\nabla_\alpha uD^2b_R\overline{\nabla_\alpha u}\,dx
&=&\int_{\R^2}(\nabla_\alpha u)_l\Big[\Big(\frac{\delta_{lk}}r-\frac{x_lx_k}{r^3}\Big)b_R'+\frac{x_lx_k}{r^2}b_R''\Big](\overline{\nabla_\alpha u})_k\,dx\\
&=&\int_{\R^2}|\nabla_\alpha u|^2\frac{b_R'}r\,dx+\int_{\R^2}|x\cdot\nabla_\alpha u|^2(\frac{b_R''}{r^2}-\frac{b_R'}{r^3})\,dx\\
&\leq&\int_{\R^2}|\nabla_\alpha u|^2\frac{b_R'}r\,dx+\int_{\R^2}(\frac{|x\cdot\nabla_\alpha u|}r)^2(1-\frac{b_R'}r)\,dx\\
&\leq& \int_{\R^2}|\nabla_\alpha u|^2\frac{b_R'}r\,dx+\int_{\R^2}|\nabla_\alpha u|^2(1-\frac{b_R'}r)\,dx\\
&=&\int_{\R^2}|\nabla_\alpha u|^2\,dx.
\end{eqnarray*}
Let $V_R:=V_{b_R}$ be given by \eqref{Vb}. By \eqref{V-b} and using the above estimates together with the standard Gagliardo-Nirenberg inequality, we infer that 
\begin{eqnarray}
\nonumber
V_{R}''(t)
&\leq&-\int_{\R^2}\Delta^2b_R|u|^2\,dx+4\int_{\R^2}|\nabla_\alpha u|^2\,dx\\
\nonumber
&-&\frac{2(p-1)}{p+1}\int_{\R^2}\Delta b_R|x|^{-\varrho}|u|^{p+1}\,dx+\frac{4}{p+1}\int_{\R^2}\nabla b_R\cdot\nabla(|x|^{-\varrho})|u|^{p+1}\,dx\\
\nonumber
&\leq&\frac c{R^2}\|u\|^2+4\int_{\R^2}|\nabla_\alpha u|^2\,dx
-\frac{4B}{p+1}\int_{\R^2}|x|^{-\varrho}|u|^{p+1}\,dx+c\int_{\{|x|>R\}}|x|^{-\varrho}|u|^{p+1}\,dx\\
\label{V-R-t}
&\lesssim&\mathcal Q[u(t)]+R^{-2}+R^{-\varrho}\|\nabla u(t)\|^{2(p-1)}.
\end{eqnarray}
Then the conclusion \eqref{Blow-infinity} easily follows  from \eqref{bl}. Indeed, suppose that $\displaystyle\sup_{0\leq t<T^*} \|\nabla u(t)\|<\infty$.  Then by \eqref{bl} and \eqref{V-R-t} it follows that $T^*<\infty$. Hence, $\|\nabla u(t)\|\to\infty$ as $t\to T^*$. This obviously contradicts the fact that $\displaystyle\sup_{0\leq t<T^*} \|\nabla u(t)\|<\infty$ and the proof is completed.
\section{Proof of Corollary \ref{t2}}
\label{S7}
Recall that $\phi$ stands for a ground state solution to \eqref{gr}.
\subsection{Proof of the scattering part in Corollary \ref{t2}}
This part follows by Theorem \ref{t1} with the next result. Indeed, the classical scattering condition below the ground state threshold is stronger than \eqref{ss1}. 
\begin{lem}
Suppose that assumptions \eqref{t11} and \eqref{t12} hold. Then, there exists $\varepsilon>0$ such that \eqref{1} is satisfied. 
\end{lem}
\begin{proof}
 Take the real function $g:t\mapsto t^2-\frac{2{\tt K}_{opt}}{p+1}t^{B}$ and compute using the identity $A+2\lambda_c=B\lambda_c$,
\begin{eqnarray*}
E[u][M[u]]^{\lambda_c}
&\geq&\|u\|_{\dot H^1_\alpha }^2\|u\|^{2\lambda_c}-\frac{2{\tt K}_{opt}}{p+1}\|u\|^{A+2\lambda_c}\|u\|_{\dot H^1_\alpha }^B\\
&=&g(\|u\|_{\dot H^1_\alpha }\|u\|^{\lambda_c}),
\end{eqnarray*}
where the optimal constant ${\tt K}_{opt}$ is given by \eqref{ineq}.
Now, with Pohozaev identities and the conservation laws, one has for some $0<\varepsilon<1$, 
\begin{eqnarray*}
g(\|u\|_{\dot H^1_\alpha }\|u\|^{\lambda_c})
&\leq&E[u][M[u]]^{\lambda_c}\\
&<&(1-\varepsilon)E[\phi][M[\phi]]^{\lambda_c}\\
&=&(1-\varepsilon)g(\|\phi\|_{\dot H^1_\alpha }\|\phi\|^{\lambda_c}).
\end{eqnarray*}

Thus, with time continuity, \eqref{t12} is invariant under the flow \eqref{S} and $T^*=\infty$. 
Moreover, by Pohozaev identities, one writes
$$E[\phi][M[\phi]]^{\lambda_c}=\frac{B-2}{B}\Big(\|\phi\|_{\dot H^1_\alpha }\|\phi\|^{\lambda_c}\Big)^2=\frac{{\tt K}_{opt}(B-2)}{p+1}\Big(\|\phi\|_{\dot H^1_\alpha }\|\phi\|^{\lambda_c}\Big)^B$$
and so
$$1-\varepsilon\geq\frac B{B-2}\Big(\frac{\|u\|_{\dot H^1_\alpha }\|u\|^{\lambda_c}}{\|\phi\|_{\dot H^1_\alpha }\|\phi\|^{\lambda_c}}\Big)^2-\frac2{B-2}\Big(\frac{\|u\|_{\dot H^1_\alpha }\|u\|^{\lambda_c}}{\|\phi\|_{\dot H^1_\alpha }\|\phi\|^{\lambda_c}}\Big)^B.$$
Following the variations of $t\mapsto\frac B{B-2}t^2-\frac2{B-2}t^B$ via the assumption \eqref{t12} and a continuity argument, there is a real number denoted also by $0<\varepsilon<1$, such that
$$\|u(t)\|_{\dot H^1_\alpha }\|u(t)\|^{\lambda_c}\leq (1-\varepsilon)\|\phi\|_{\dot H^1_\alpha }\|\phi\|^{\lambda_c}\quad\mbox{on}\quad \R.$$
Now, by the last line and Pohozaev identities, for some real number denoted also by $0<\varepsilon<1$,
\begin{eqnarray*}
\mathcal P[u][M[u]]^{\lambda_c}
&\leq&{\tt K}_{opt}\|u\|_{\dot H^1_\alpha }^B\|u\|^{A+2\lambda_c}\\
&\leq&{\tt K}_{opt}(1-\varepsilon)(\|\phi\|_{\dot H^1_\alpha }\|\phi\|^{\lambda_c})^B\\
&\leq&(1-\varepsilon)\frac{p+1}B(\|\phi\|_{\dot H^1_\alpha }\|\phi\|^{\lambda_c})^2\\
&\leq&(1-\varepsilon)\mathcal P[\phi]M[\phi]^{\lambda_c}.
\end{eqnarray*}
This finishes the proof.
\end{proof}
\subsection{Proof of the blow-up part in Corollary \ref{t2}}
Assume that \eqref{t11} and \eqref{t13} are satisfied. Let us prove that $\mathcal{GM}[u(t)]>1$ on $[0,T^*)$ where $\mathcal{GM}[u(t)]$ is given by \eqref{M-G}. 
\begin{lem}\label{stb}
The conditions \eqref{t11} and \eqref{t13} are stable under the flow of \eqref{S}.
\end{lem}
\begin{proof}
Taking into account Proposition \ref{gag}, one writes
\begin{eqnarray*}
M[\phi]^{\lambda_c}E[\phi]
&>&(1+\varepsilon)\|u\|^{2\lambda_c}\Big(\|u\|_{\dot H^1_\alpha }^2-\frac2{p+1}\mathcal P[u]\Big)\\
&>&(1+\varepsilon)\|u\|^{2\lambda_c}\Big(\|u\|_{\dot H^1_\alpha }^2-\frac{2{\tt K}_{opt} }{p+1}\|u\|^A\|u\|^B_{\dot H^1_\alpha }\Big)\\
&:=&(1+\varepsilon)f(\|u\|^{2\lambda_c}\|u\|_{\dot H^1_\alpha }).
\end{eqnarray*}
Compute also by Proposition \ref{gag},
\begin{eqnarray*}
f(\|\phi\|^{2\lambda_c}\|\phi\|_{\dot H^1_\alpha })
&=&\|\phi\|^{2\lambda_c}\Big(\|\phi\|_{\dot H^1_\alpha }^2-\frac{2{\tt K}_{opt} }{p+1}\|\phi\|^A\|\phi\|^B_{\dot H^1_\alpha }\Big)\\
&=&\|\phi\|^{2\lambda_c}\Big(\|\phi\|_{\dot H^1_\alpha }^2-\frac{2}{p+1}\mathcal P[\phi]\Big)\\
&=&M[\phi]^{\lambda_c}E[\phi].
\end{eqnarray*}
Thus, $f(\|\phi\|^{2\lambda_c}\|\phi\|_{\dot H^1_\alpha })>(1+\varepsilon)f(\|u\|^{2\lambda_c}\|u\|_{\dot H^1_\alpha })$. The proof follows by a continuity argument.
\end{proof}
By Pohozaev's identity, we get $B\,E[\phi]=(B-2)\|\phi\|_{\dot H^1_\alpha }^2$. Therefore
\begin{eqnarray*}
\mathcal Q[u][M[u]]^{\lambda_c}
&=&\Big(\| u\|_{\dot H^1_\alpha }^2-\frac{B}{p+1}\mathcal P[u]\Big)[M[u]]^{\lambda_c}\\
&=&\frac{B}{2}E[u][M[u]]^{\lambda_c}-(\frac{B}{2}-1)\| u\|_{\dot H^1_\alpha }^2[M[u]]^{\lambda_c}\\
&\leq&\frac{B}{2}(1-\varepsilon)E[\phi][M[\phi]]^{\lambda_c}-(\frac{B}{2}-1)\|\phi\|_{\dot H^1_\alpha }^2[M[\phi]]^{\lambda_c}\\
&\leq&-\varepsilon\|\phi\|_{\dot H^1_\alpha }^2[M[\phi]]^{\lambda_c}.
\end{eqnarray*}
The proof follows by the use of Theorem \ref{t1}.



\end{document}